\documentclass[]{elsarticle}
\usepackage{amsmath,amsfonts,amsthm,graphicx,color,geometry,multicol,caption}
\usepackage{tabularx}
\usepackage[utf8]{inputenc} 
\usepackage{enumitem}
\usepackage[]{subfig}
\usepackage{mathrsfs}
\usepackage{tikz}
\usetikzlibrary{matrix}
\usetikzlibrary{decorations.pathreplacing}
\usepackage{multirow}

\usepackage{mathtools}
\mathtoolsset{showonlyrefs=true}

\newcommand{\eq}[1]{ 
\begin{equation}
	\begin{split}
		#1 
	\end{split}
\end{equation}
} 
\newcommand{\eqn}[1]{ 
\begin{equation}
	\begin{split}
		#1 \nonumber 
	\end{split}
\end{equation}
} 

\DeclareMathOperator{\im}{Im}
\DeclareMathOperator{\cor}{corank}
\newcommand{\ve}{\varepsilon}

\newcommand{\oX}{\overline{X}} 
\newcommand{\oY}{\overline{Y}} 
\newcommand{\oz}{\bar{z}}

\newcommand{\R}{\mathbb{R}} 
 
\newcommand{\sfrac}[2]{\text{\small{$\frac{#1}{#2}$}}} 
 
\newcommand{\parc}[1]{\dfrac{\partial}{\partial #1}} 
\newcommand{\parcs}[2]{\dfrac{\partial #1}{\partial #2}} 
\newcommand{\sv}{S_V} 
\newcommand{\svm}{S_{V,min}} 
 
\newcommand{\E}[0]{\mathcal{E}}

\newcommand{\B}[0]{\mathcal{B}} 
\newcommand{\C}[0]{\mathcal{C}} 
\newcommand{\Ci}[0]{\mathcal{C}^\infty}

\newtheorem{example}{Example}[section]
\newtheorem{remark}{Remark}[section] 
\newtheorem{theorem}{Theorem}[section] 
\newtheorem{lemma}{Lemma}[section] 
\newtheorem{proposition}{Proposition}[section] 
\newtheorem{corollary}{Corollary}[section]
\newtheorem{definition}{Definition}[section]
\numberwithin{equation}{section}

\everymath{\displaystyle}
\begin{document}
 
\begin{frontmatter}
	\title{Polynomial normal forms of constrained differential equations with three parameters.}
	\author[jbi]{H. Jard\'on-Kojakhmetov\corref{cor1}}
	\ead{h.jardon.kojakhmetov@rug.nl}
	\author[jbi]{Henk W. Broer}
	\ead{h.w.broer@rug.nl}
	\address[jbi]{Johann Bernoulli Institute for Mathematics and Computer Science University of Groningen, P.O. Box 407, 9700 AK, Groningen, The Netherlands}
	\journal{Journal of Differential Equations.}
	\cortext[cor1]{Corresponding author.}
	\begin{abstract}
	    We study generic constrained differential equations (CDEs) with three parameters, thereby extending Takens's classification of singularities of such equations. In this approach, the singularities analyzed are the Swallowtail, the Hyperbolic, and the Elliptic Umbilics. We provide polynomial local normal forms of CDEs under topological equivalence. Generic CDEs are important in the study of slow-fast (SF) systems. Many properties and the characteristic behavior of the solutions of SF systems can be inferred from the corresponding CDE. Therefore, the results of this paper show a first approximation of the flow of generic SF systems with three slow variables.
	\end{abstract}
	\begin{keyword}
		Constrained Differential Equations, Slow-Fast systems, Normal Forms, Catastrophe Theory.
	\end{keyword}
\end{frontmatter}

\tableofcontents 

\section{Introduction}
	\renewcommand*{\a}{\alpha}
    The present document studies \emph{constrained differential equations} (CDEs) with three parameters. The main motivation comes from \emph{slow-fast} systems, which are usually given as
    \eq{\label{eq:intro_slow}
    \ve\dot x &= f(x,\a,\ve)\\
    \dot    \a &= g(x,\a,\ve),
    }
    where $x\in\R^n$ represents states of a process, $\a\in\R^m$ denotes control parameters, and $\ve>0$ is a small constant. Mathematical equations as \eqref{eq:intro_slow} are often used to model phenomena with two time scales. A constrained differential equation is the limit $\ve=0$ of \eqref{eq:intro_slow}, that is
    \eq{\label{eq:intro_cde}
    0 &= f(x,\a,0)\\
    \dot    \a &= g(x,\a,0).
    }
    We assume throughout the rest of the text that the functions $f(\cdot)$ and $g(\cdot)$ are $\Ci$ smooth (all partial derivatives exist and are continuous). From \eqref{eq:intro_slow} one can observe that whenever $f(\cdot)\neq 0$, the smaller $\ve$ is, the faster $x$ evolves with respect to $\a$. Therefore, in the context of SF systems, the coordinates $x$ and $\a$ receive the name of \emph{fast} and \emph{slow} respectively. Defining the new time parameter $\tau=t/\ve$, the system \eqref{eq:intro_slow} can be rewritten as
    \eq{\label{eq:intro_fast}
    x' &= f(x,\a,\ve)\\
    \a' &= \ve g(x,\a,\ve),
    }
    where $'$ denotes derivative with respect to the fast time $\tau$. Systems \eqref{eq:intro_slow} and \eqref{eq:intro_fast} are equivalent as long as $\ve\neq 0$. In the limit $\ve=0$ the system \eqref{eq:intro_fast} reads
    \eq{\label{eq:intro_layer}
    x' &= f(x,\a,0)\\
    \a' &= 0,
    }
    and it is called \emph{the layer equation}. A first approximation of the slow-fast dynamics of \eqref{eq:intro_slow} (or \eqref{eq:intro_fast}) is given by studying both \eqref{eq:intro_cde} and \eqref{eq:intro_layer}. 
	
	\begin{remark}\label{rrev}\leavevmode
		\begin{itemize}
			\item There are some important features, such as canards, of slow fast systems that can not be studied in the limit $\ve=0$ \cite{Benoit, DumRou2, Szmolyan2001419}. However, having a \emph{generic model} of the constrained equation is important in order to study the complicated phenomena that related SF systems exhibit.
			\item We are interested in the case where the layer equation (or fast dynamics) is given as a gradient system. More specifically, we assume that there exists a smooth $m$-parameter family $V:\R^n\times\R^m\to\R$ such that
    \eq{\label{com}
    f(x,\a,0)=\parcs{V}{x}(x,\a).
    }
	Although not every slow fast system satisfies \eqref{com}, there is a motivation behind this. From the mathematical point of view, it is interesting to see how the classification of singularities of smooth maps can be used to find normal forms. It is precisely the purpose of this document to exploit such idea. Applications are also an important motivation. Two remarkable features of SF systems, canards and relaxation oscillations are found in models where $f(x,\a,0)$ is locally a fold singularity \cite{Krupa1,Krupa2}.  Furthermore, there are interesting real life phenomena which can indeed be modeled by systems satisfying \eqref{com}. Two examples are shown in section \ref{sec:motivation} and some more can be consulted in \cite{gucwa2009geometric,KosiukS11,milik2001multiple,milik1998geometry}.
		\end{itemize}
		
	\end{remark}
	
    The family $V$ is called \emph{potential function}. By such consideration, we define the constraint manifold $\sv$ as the critical set of $V$, this is
    \eq{\label{eq:intro_sv}
    \sv=\left\{ (x,\a)\in\R^n\times\R^m \, | \, \parcs{V}{x}(x,\a)=0\right\}.
    }
     Observe that the set $\sv$ serves as the \emph{phase space} of the CDE \eqref{eq:intro_cde}, and as the set of equilibrium points of the layer equation \eqref{eq:intro_layer}. We can roughly interpret the dynamics of a CDE as follows. Let a potential function $V$ be given. If the initial condition $(x_0,\a_0)\notin\sv$, $x$ has to adjust infinitely fast (according to \eqref{eq:intro_layer}) to satisfy the constraint $\sv$. This infinitely fast behavior occurs along the so called \emph{fast foliation}, which is a family of $n$-dimensional hyperplanes parallel to the $(x,0)$ space. Once the constraint is satisfied, the dynamics follow \eqref{eq:intro_cde}. Naturally, $\sv$ does not need to be a regular manifold. It may very well happen that the potential function $V$ has degenerate critical points. In fact, it is in such situation where the most interesting phenomena appear. Two classical examples are given in sections \ref{sub:zeeman_s_heart_beat_model} and \ref{sub:zeeman_s_nerve_impulse_model}. For an illustration of the previous description see figure \ref{fig:intro_schematic}.\\
    
    \begin{figure}[htbp]\centering
        \includegraphics[]{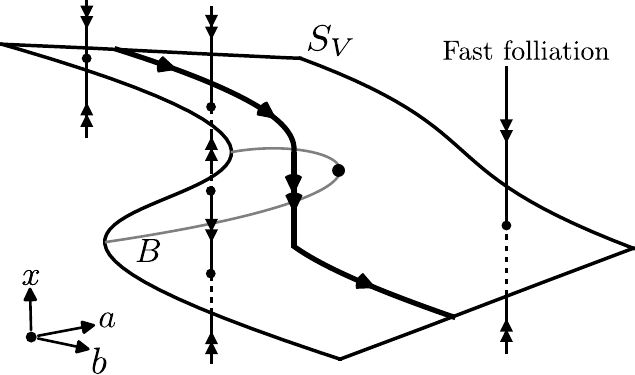}
        \caption{Schematic representation of the solutions of a constrained differential equation with one state variable \ensuremath{\left(x\right)} and two control parameters \ensuremath{\left(a,b\right)}. If the initial conditions do not lie within the critical set \ensuremath{\sv}, then there is an infinitely fast transition towards \ensuremath{\sv} according to the layer equation \protect\eqref{eq:intro_layer}. Once the constraint \ensuremath{\sv} is satisfied, the dynamics are governed by the CDE \protect\eqref{eq:intro_cde}. The phase space is then the manifold \ensuremath{\sv}. Such manifold may have singularities, which consist of points in \ensuremath{\sv} tangent to the fast foliation. The set of such tangent points is denoted by \ensuremath{B}. At such points, the trajectories may jump to another stable part of \ensuremath{\sv} or they may indefinitely follow the fast foliation.}
        \label{fig:intro_schematic}
    \end{figure} 
    
    In the context of CDEs, one is interested on the description of the \emph{local} behavior of \eqref{eq:intro_cde} in an arbitrarily small neighborhood of a singularity of the potential $V$. We assume that such singularity is located at the origin. Formally speaking, we consider germs \cite{Arnold_singularities, Brocker} of functions $V$ at the origin. Therefore, in the rest of the paper whenever we write a function $V:\R^n\times\R^m\to \R$ we actually mean that $V$ is the preferred representative of the germ of $V$ at the origin. Given such potential, then one studies the types of vector fields that are likely to occur. 
    
    \begin{remark} As we detail below, a normal form of a CDE is given by a generic\footnote{The term generic stands for maps satisfying Thom's transversality theorem. See theorem \ref{teo:trans} in section \ref{sec:cde}.} local potential function $V$, and by a member of an equivalence class of vector fields (see section \ref{sec:cde}). That is, an important element on the analysis of singularities of CDEs is the classification of families $V:\R^n\times\R^m\to\R$. For sufficiently small number of parameters, such classification problem is known as elementary catastrophe theory (see section \ref{sec:cat}).
    \end{remark}
    
   Constrained equations \eqref{eq:intro_cde} are a first approximation of the slow dynamics of a slow-fast system \eqref{eq:intro_slow}.  Therefore, normal forms of CDE play an important role in understanding the overall dynamics of the corresponding SF system. The latter type of equation with one (Fold) and two (Cusp) slow variables have been studied in \cite{DumRou1, Krupa1, Krupa2, vanGils} and in \cite{BKK} respectively. The main contribution of this paper consists on a list of normal forms of CDEs with three parameters (see theorem \ref{teo:main}). This means that up to an $\ve=0$ approximation, we also provide a description of generic slow-fast systems with three slow variables. Moreover, the methodology and ideas presented in the main part of this article can be used to provide topological normal forms of CDEs with ``more complicated singularities'', which in our context amounts to more degenerate potential $V$ or more, also degenerate, fast variables. An example would be the topological classification of CDEs with four parameters.\\
	
    The present document is arranged as follows. In section \ref{sec:cat} we briefly recall the basic concepts of elementary catastrophe theory. After this, in section \ref{sec:motivation} we present a couple of classical examples of slow-fast systems used to roughly model real life phenomena. Next, in section \ref{sec:cde} we review  the formal definitions, and the main results of CDE theory \cite{Takens1}. Afterwards, in section \ref{sec:nf} we present a geometric analysis of constrained differential equations with three parameters focusing on the catastrophes defining the generic potential functions and their influence in the type of vector fields that one may generically encounter. Once we provide sufficient geometric insight of the problem, we present our results in sections \ref{sec:main_theorem} and \ref{sec:jumps} followed by the corresponding proofs. For completeness, in the appendix we include some background theory to which we refer in the main text.

\section{Elementary catastrophe Theory}\label{sec:cat}
    \textit{Catastrophe theory} has its origins in the 1960's with the work of Ren\'e Thom \cite{Thom2,Thom1,Thom3}. One of its goals was to qualitatively study the sudden (or catastrophic) way in which solutions of biological systems change upon a small variation of parameters. The most basic setting of this theory is called \textit{elementary} catastrophe theory \cite{Golubitsky,Poston,Stewart1}. It is concerned with gradient dynamical systems
    \eq{
    \dot x=-\parc{x} V(x,\a).
    \label{eq:grad1}
    }
    The variables $x\in\R^n$ represent the \textit{states} or the measurable quantities of a certain process, and $\a\in\R^m$ represent \textit{control parameters}. One concern is to find equilibria of \eqref{eq:grad1}, this is, to solve
    \eq{
    \parc{x} V(x,\a)=0.
    \label{eq:grad2}
    }
    In mathematical terminology, one is interested in the qualitative behavior of the solutions $x$ of \eqref{eq:grad2} as the parameters $\a$ change. It is also interesting to know to what extent different functions $V$ may show the same topology (or the same local behavior). These ideas led to the topological classification of families of degenerate functions $V(x,\a):\R^n \times \R^m\to \R$ for $m\leq 4$, which is known as the ``seven elementary catastrophes", see table \ref{cats}.

    \begin{theorem}[Thom's classification theorem \cite{Brocker}]  Let $V(x,\a):\R^n \times \R^m\to \R$ be an $m-$parameter family of smooth functions $V(x,0):\R^n \to \R$, with $m\leq 4$. If $V(x,\a)$ is generic then it is right-equivalent (up to multiplication by $\pm 1$, up to addition of Morse functions and up to addition of functions on the parameters) to one of the forms shown in table \ref{cats}.

    \begin{table}[htbp]\small
    \begin{center}
    \begin{tabular}{l|l|c}
    Name & $V(x,\a)$ & Codimension\\[1ex]
    \hline
    Non-critical & $x$  &   \\[1ex]
    Non-degenerate (Morse) & $x^2$ & $0$ \\[1ex]
    \hline
    Fold & $\tfrac{1}{3}x^3+ax$ & $1$\\[1ex]
    Cusp & $\tfrac{1}{4}x^4+\tfrac{1}{2}ax^2+bx$ & $2$\\[1ex]
    Swallowtail & $\tfrac{1}{5}x^5+\tfrac{1}{3}ax^3+\tfrac{1}{2}bx^2+cx$ & $3$ \\[1ex]
    Elliptic Umbilic & $x^3-3xy^2+a(x^2+y^2)+bx+cy$ & $3$ \\[1ex]
    Hyperbollic Umbilic & $x^3+y^3+axy+bx+cy$ & $3$  \\[1ex]
    Butterfly & $\tfrac{1}{6}x^6+\tfrac{1}{4}ax^4+\tfrac{1}{3}bx^3+\tfrac{1}{2}cx^2+dx$ & $4$ \\[1ex]
    Parabolic Umbilic & $x^2y+y^4+ax^2+by^2+cx+dy$ & $4$ \\
    \end{tabular} 
    \end{center}
    \caption{Thom's classification of families of functions for $m\leq 4$. Each elementary catastrophe is a structurally stable $m$-parameter unfolding of the germ $V(x,0)$.}
    \label{cats} 
    \end{table}
	
    \end{theorem}
    \begin{remark}  Loosely speaking, the codimension of a singularity is the minimal number of parameters $m$ for which a singularity persistently occurs in an $m-$parameter family of functions. In this paper we focus on constrained differential equations \eqref{eq:intro_cde} written as
    	\eq{
    		0      &= -\parcs{V}{x}(x,\a) \\
    		\dot \a &= g(x,\a),
    		\label{eq:cde2}
    	}
    where $\a\in\R^3$, and therefore $V(x,\a)$ is any of the codimension $3$ catastrophes of table \ref{cats}. For each of such items, we provide polynomial local normal forms (modulo topological equivalence) of the vector field $g(x,\a)\parc{a}$.

    \end{remark}

\section{Motivating examples}\label{sec:motivation}
    In this section we review two classical examples of natural phenomena that can be qualitatively understood by means of elementary catastrophe theory, and that are modeled by slow-fast systems. These applications were thoroughly studied by Zeeman \cite{Zeeman1}. His interest for using this theory was that it enables a qualitative description of the local dynamics of a biological system instead of modeling the complicated biochemical processes involved. These examples also serve to understand the way the CDEs and SF systems relate to each other.

    \subsection{Zeeman's heartbeat model} 
    \label{sub:zeeman_s_heart_beat_model}

    The simplified heart is considered to have two (measurable) states. The \textit{diastole} which corresponds to a relaxed state of the heart's muscle fiber, and \textit{systole} which stands for the contracted state. When a heart stops beating it does so in relaxed state, an equilibrium state. There is an electrochemical wave that makes the heart contract into systole. When such wave reaches a certain threshold, it triggers a sudden contraction of the heart fibers: a catastrophe occurs. After this, the heart remains in systole for a certain amount of time (larger in comparison to the contraction-relaxation time) and then rapidly returns to diastole. A mathematical local representation of the behavior just explained is given by 
    \eq{
    \ve\dot x&=-(x^3-x+b)\\
    \dot b&=x-x_0,
    \label{eq:heart}
    }   
    where $x, \; b \in \R$. Observe the similarity of \eqref{eq:heart} with a Van der Pol oscillator with small damping \cite{vanderpol_heart}. The variable $x$ models the length of the muscle fiber, $b$ corresponds to an electrochemical control variable and $x_0>\frac{1}{\sqrt{3}}$ represents the threshold. In the limit $\ve=0$  we obtain the CDE
    \eq{
    0&=-(x^3-x+b)\\
    \dot b&=x-x_0.
    \label{eq:heart_cons}
    } 
    The potential function $V$ is a section of the cusp catastrophe, see table \ref{cats} and note that $a=-1$. The constraint manifold is defined by $S_V=\left\{ (x,b)\in\R\times\R \; | \; x^3-x+b=0 \right\}$. Observe that there are two fold points defining the singularity set.
    \eq{
    B=\left\{ (b,x)\in\R^2 \, | \, 3x^2-1=0 \right\},
    }

    this is
    \eq{
    B=\left( \frac{2}{3\sqrt{3}}, \frac{1}{\sqrt{3}}\right)\bigcup\left( -\frac{2}{3\sqrt{3}}, -\frac{1}{\sqrt{3}} \right).
    }

     The set $B$ corresponds singularities of $S_V$, where the fast foliation is tangent to the curve $S_V$. At such points, the trajectory has a sudden change of behavior, it jumps.  A schematic of the dynamics of \eqref{eq:heart_cons} is shown in figure \ref{fig:heart}.\\
    \begin{figure}[hbtp]
    \centering
    \includegraphics[scale=0.5]{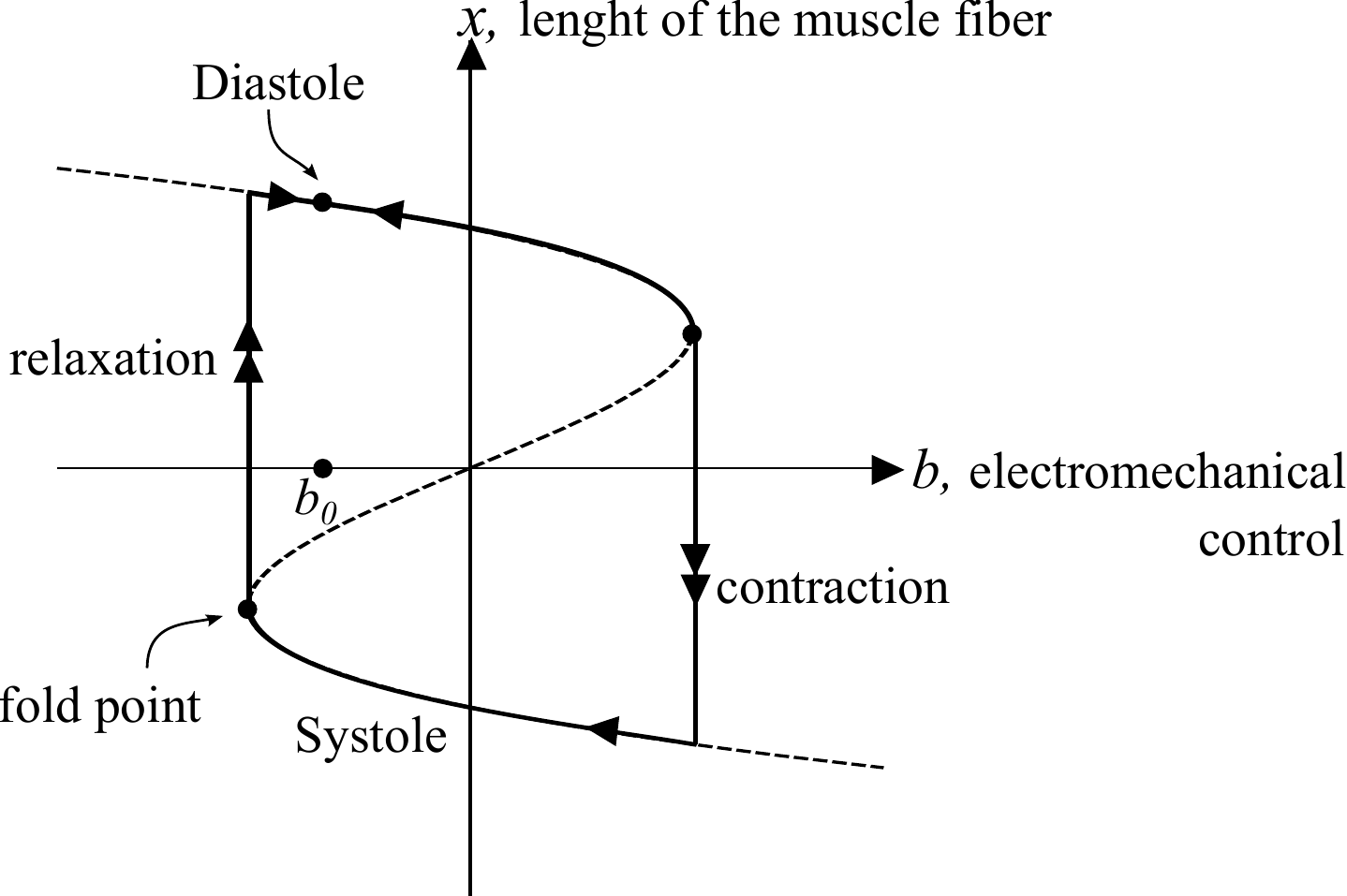}
    \caption{Dynamics of the simplified heartbeat model \protect\eqref{eq:heart_cons}. A pacemaker controls the value of  the parameter \ensuremath{b} changing its value from \ensuremath{b_0} up to an adequate threshold such that the action of contraction is triggered. Such contraction (and relaxation) is modeled by a fast transition between the two stable branches of the curve \ensuremath{S_V}. } 
    \label{fig:heart}
    \end{figure}
	
	For sufficiently small $\ve> 0$, the trajectories of \eqref{eq:heart} are close to those of \eqref{eq:heart_cons}. It is one of the goals of the theory of SF systems to make precise the notion of closeness mentioned above, especially in the neighborhood of singular points (see for example \cite{Fenichel, DumRou1}).

    \subsection{Zeeman's nerve impulse model} 
    \label{sub:zeeman_s_nerve_impulse_model}

    This model qualitatively describes the local and simplified behavior of a neuron when transmitting information through its axon, see \cite{Zeeman1} for details and compare also with the Hodgkin-Huxley model \cite{HH}. Qualitatively speaking, there are three important components on this process: the concentration of Sodium (Na) and Potassium (K), and the Voltage potential (V) in the wall of the axon. As information is being transmitted, there is a slow and smooth change of the Voltage and of the concentration of Potassium but a rather sudden change in the concentration of Sodium. Another local characteristic is that the return to the equilibrium state, when there is no transmission, is slow and smooth. The three variables mentioned behave qualitatively as shown in the figure \ref{fig:nerve_signals}. \\

    \begin{figure}[htbp]
    \centering
    	\includegraphics[scale=0.8]{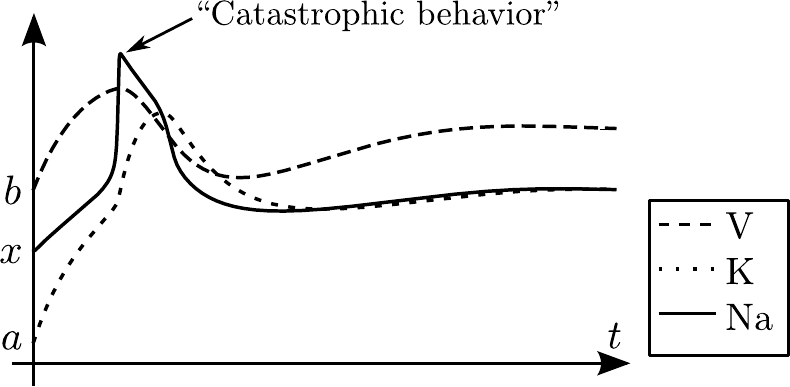}
    	\caption{\cite{Zeeman1} A qualitative picture of the three variables involved in the local model of the nerve impulse. The signal \ensuremath{V} represents the potential of the axon walls. The signals of \ensuremath{Na} and \ensuremath{K} represent the conductance of Sodium and Potassium respectively. Observe that a characteristic property is the sudden and rapid change of the Sodium conductance followed by a smooth and slow return to its equilibrium state. See \protect\cite{Zeeman1}, where a qualitatively similar graph is plotted from measured data. }
    	\label{fig:nerve_signals}
    \end{figure}

    A mathematical model that roughly describes the nerve impulse process is given by
    \eq{ \label{eq:nerve_sf}
    	\ve \dot x &= -(x^3+ax+b)\\
    	\dot a     &= -2(a+x)\\
    	\dot b     &= -1-a.
    }
    The corresponding constrained differential equation reads
    \eq{ \label{eq:nerve_cde}
    	0 &= -(x^3+ax+b)\\
    	\dot a     &= -2(a+x)\\
    	\dot b     &= -1-a.
    }
    The defining potential function is $V=\tfrac{1}{4}x^4+\tfrac{1}{2}ax^2+bx$, that is the cusp catastrophe of table \ref{cats}. The constraint manifold is defined as
	 \eqn{S_V=\left\{(x,a,b)\in\R\times\R^2\; | \;-(x^3+ax+b)=0\right\},
	 } 
	 and is the critical set of $V$. Recall that $\sv$ serves as the phase space of the flow of \eqref{eq:nerve_cde}. The attracting part of the manifold $\sv$, denoted by $\svm$, is given by points where $D^2_xV>0$, this is 
	\eqn{
	\svm=\left\{ (x,a,b)\in\sv\,|\, 3x^2+a>0\right\},
	}
	
    If we restrict the coordinates to $\sv$, we can perform the transformation $(a,b)\mapsto (a,-x^3-ax)$, which allows us to rewrite \eqref{eq:nerve_cde} as the planar system
    \eq{ \label{eq:nerve_des1}
    	\dot a &= -2(a+x)\\
    	\dot x &= \frac{1+a+2(a+x)x}{3x^2+a}.
    }
    The vector field \eqref{eq:nerve_des1} is not smooth. It is not well defined at the singular set \eqn{B=\left\{ (x,a)\in\sv\,|\, 3x^2+a=0\right\}.} However outside $B$,  the flow of \eqref{eq:nerve_des1} is equivalent to the flow of
    \eq{ \label{eq:nerve_des2}
    	\dot a &= -2(3x^2+a)(a+x)\\
    	\dot x &= 1+a+2(a+x)x.
    }
    The vector field \eqref{eq:nerve_des2} receives the name of \emph{the desingularized vector field}. Note that \eqref{eq:nerve_des2}  is smooth and is defined for all $(x,a)\in\R^2$. The importance of \eqref{eq:nerve_des2} lays in the fact that one obtains the solutions of the CDE \eqref{eq:nerve_cde} from the integral curves of \eqref{eq:nerve_des2}. The general reduction process through which we obtain the desingularized vector field is described in section \ref{sec:desingularization}.\\
    
    Observe that \eqref{eq:nerve_des2} has equilibrium points $(a,x)$ as follows.
    \begin{itemize}
    	\item $p_a = (-1,1)$, which is a regular equilibrium point.
    	\item $p_f = \left( -\frac{3}{4}, \frac{1}{2}  \right)$, which is contained in the fold line, thus receives the name folded singularity. 
    \end{itemize}

    Furthermore, $p_f$ is a saddle point, whence it is called folded-saddle singularity. Observe in figure \ref{fig:nerve} the phase portrait of \eqref{eq:nerve_des2} and note the smooth return of some trajectories and compare with the heartbeat model where this effect does not occur. \\

    Once \eqref{eq:nerve_des2} is better understood, we are able to give a qualitative picture of the flow of \eqref{eq:nerve_cde} recalling that to obtain \eqref{eq:nerve_des2} we performed the change of variables $(a,b)\mapsto (a,-x^3-ax)$, and we scaled by the factor $3x^2-a$.  We show in figure \ref{fig:nerve} the phase portraits of desingularized vector field \eqref{eq:nerve_des2} and of the CDE \eqref{eq:nerve_cde}.
    
    \begin{figure}[ht]\centering
        \begin{tikzpicture}[]
            \node [scale=0.75] at (0,0) {\includegraphics{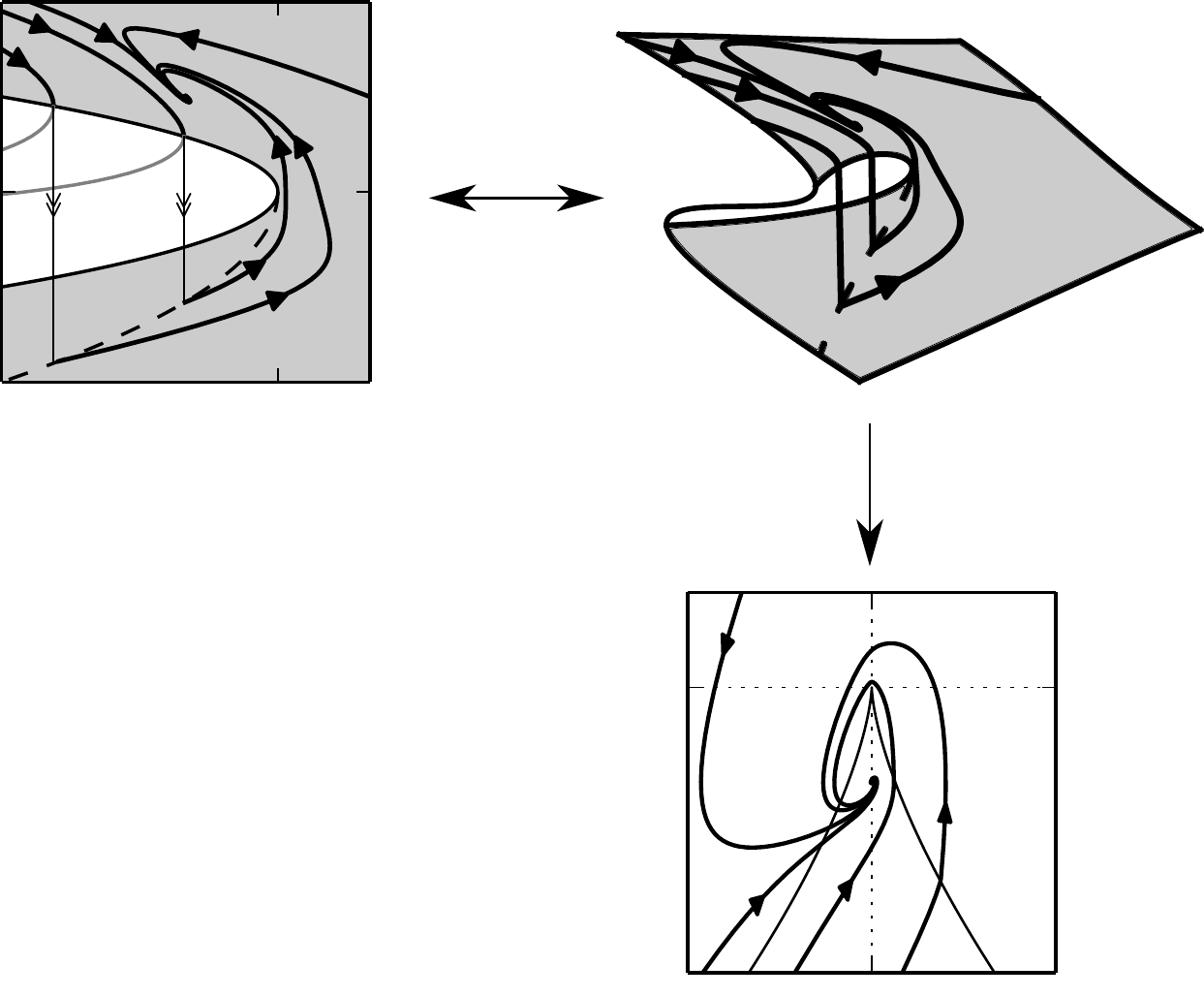}};
        \end{tikzpicture}
        \begin{tikzpicture}[overlay]
            \node at (0,6) {$S_V$};
            \node at (-4.6,6.05) {$B$};
            \node at (-2.6,4) {$\pi$};
            \node at (-5.55,6.6) {$\mathcal D$};
            \node at (-2.7,0) {$b$};
            \node at (-4.6,2.5) {$a$};
            \node at (-10.2,6.4) {$x$};
            \node at (-7.9,4.6) {$a$};
        \end{tikzpicture}
        \begin{tikzpicture}[overlay]
            \draw [decorate,decoration={brace,amplitude=10pt,mirror,raise=2pt},yshift=0pt]
            (-1.4,.2) -- (-1.4,3.4) node [black,midway,xshift=2.1cm] {\small
            $\tilde X:\begin{cases}
        	\dot a     &= -2(a+x)\\
        	\dot b     &= -1-a, \\
			&\text{Satisfying}\\
			0 &= x^3+ax+b
            \end{cases}$};
            \draw [decorate,decoration={brace,amplitude=10pt,mirror,raise=2pt},yshift=0pt]
            (-9.9,4.6) -- (-6.85,4.6) node [black,midway,yshift=-1cm] {\small
            $\overline X:\begin{cases}
        	\dot a &= -2(3x^2+a)(a+x)\\
        	\dot x &= 1+a+2(a+x)x.
            \end{cases}$};
        \end{tikzpicture}
    	\caption{Top right: phase portrait of the CDE \protect\eqref{eq:nerve_cde}. The manifold \ensuremath{S_V} serves as the phase space of the corresponding flow. The shaded region is the attracting part of the constraint manifold, that is \ensuremath{\svm}. Top left: phase portrait of the desingularized vector field \protect\eqref{eq:nerve_des2}. In this picture, and in the rest of the document, the symbol \ensuremath{\mathcal D} denotes the desingularization process, to be detailed in section \protect\ref{sec:desingularization}. Observe that although the vector field \ensuremath{\oX} is defined for all \ensuremath{(x,a)\in\R^2} we are only interested in the region $\svm$. When the trajectories reach the singular set \ensuremath{B}, they jump to another attracting part of \ensuremath{\svm}. Bottom right: projection of the phase portrait of \protect\eqref{eq:nerve_cde} onto the parameter space. The map \ensuremath{\pi} is a smooth projection of the total space \ensuremath{(x,a,b)\in\R^3} onto the parameter space \ensuremath{(a,b)\in\R^2}.   }
    	\label{fig:nerve}
    \end{figure}
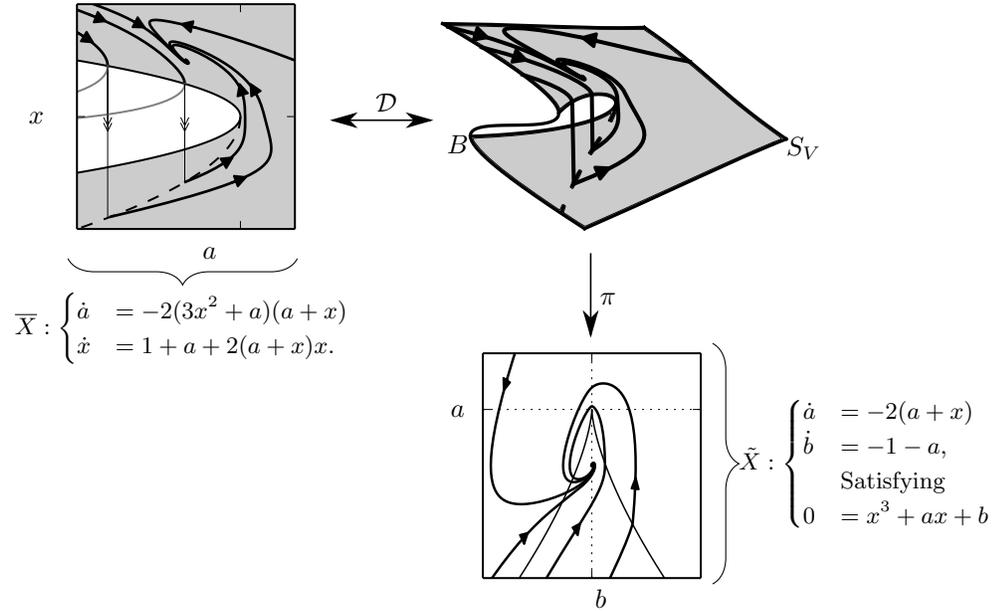
    
    \begin{remark} \label{remark_gen} Figure \ref{fig:nerve} graphically shows all the important elements in the theory of constrained differential equations. 
        \begin{itemize}
            \item The constraint manifold $\sv$ is the phase space of the flow of the CDE.
            \item The map $\pi:\R^n\times\R^m\to\R^m$ is a smooth projection from the total space onto the parameter space. The vector field induced in this space is denoted by $\tilde X$.
            \item The smooth vector field $\oX$ is obtained by desingularization, which we denote by $\mathcal D$. In the previous example such process is as follows. First one restricts the coordinates to the constraint manifold, allowing the change of coordinates $b=-3x^2-a$. Then project such restriction onto the parameter space, this is $(x,a,b)|\sv=(x,a,-3x^2-a)\mapsto (a,-3x^2-a)$. By such reparametrization we are able to compute the smooth vector field $\oX$. Observe that for points in $\svm$, the desingularization process $\mathcal D$ can be seen as a map between the solution curves of $\oX$ and those of the CDE \eqref{eq:nerve_cde}. The details of the desingularization procedure is to be given in section \ref{sec:desingularization}.
        \item The solutions of the CDE are obtained from the integral curves of the desingularized vector field $\oX$.
        \end{itemize}
    \end{remark}

\section{Constrained Differential Equations}\label{sec:cde}
\newcommand*\circled[1]{\tikz[baseline=(char.base)]{\node[shape=circle,draw,inner sep=1pt] (char) { #1};}}

    In this section we provide a brief introduction to the theory of constrained differential equations developed by Takens \cite{Takens1}. We also present  some results to be extended in the present paper. Particularly, we discuss the desingularization process, which is an important step in the study of singularities of CDEs. Next we give Takens's list of local normal forms of CDEs with two parameters. 

    \subsection{Definitions}\label{sec:definitions}
    
    \begin{definition}[\sc{Constrained Differential Equation (CDE)}]\label{def:CDE}
    	Let $\E$ and $\B$ be $\C^\infty$-manifolds, and $\pi: \E\to\B$ a $\C^\infty$-projection. \emph{A constrained differential equation on $\E$} is a pair $(V,X)$, where $V: \E\to\R$ is a $\C^\infty$-function, called \emph{potential function}, that has the following properties:
    	\begin{itemize}
    	 	\item[CDE.1 ] $V$ restricted to any fiber of $\E$ (denoted by $V|\pi^{-1}\left( \pi(e)\right)$, $e\in\E$) is proper and bounded from below,
    	 	\item[CDE.2 ] \label{eq:sv} the set $S_V=\left\{ e\in\E : V|\pi^{-1}\left( \pi(e)\right) \text{has a critical point in $e$}\right\}$, called \emph{the constraint manifold}, is locally compact in the sense: for each compact $K\subset\B$, the set $\sv\bigcap\pi^{-1}(K)$ is compact,
    	 \end{itemize} 
    	 and  $X$ is such that:
    	 \begin{itemize}
    	 	\item[CDE.3 ]  $X:\E \to T\B$ is a $\C^\infty$-map covering $\pi:\E\to\B$.
    	 \end{itemize}
    \end{definition}

     \begin{remark} $ $
     	\begin{itemize} 
     		\item $S_V$ is a smooth manifold of the same dimension as $\B$.

     		\item The covering property of $X$ means that for all $e\in\E$, the tangent vector $X(e)$ is an element of $T_{\pi(e)}\B$, the tangent space of $\B$ at the point $\pi(e)$. $T\B$ denotes the tangent bundle of $\B$. The covering property of $X$ defines a vector field $\tilde X:\B\to T\B, \tilde X=X\circ \pi^{-1}$. 

        \end{itemize}
 	
     \end{remark}
     
     \begin{definition}[\sc The set of minima]
     	The set $\svm$ is defined by
     \small\[ S_{V,min}=\left\{ e\in\E :  V|\pi^{-1}\left( \pi(e)\right) \text{has a critical point in $e$, which Hessian is positive semi-definite}\right\} \]
     \end{definition}
     
     Recall that in coordinate notation we are studying equations of the form
 	\eqn{
 		0      &= -\parcs{V}{x}(x,\a) \\
 		\dot \a &= g(x,\a),
 	}
    and therefore $\svm$ corresponds to the attracting region of $\sv$.

     \begin{definition}[\sc Solution] Let $(V,X)$ be as in definition \ref{def:CDE}. A curve $\gamma: J\to \E$, $J$ an open interval of $\R$, is a \emph{solution} of $(V,X)$ if 
     \begin{itemize}
     	\item[S1 ] $\gamma\left( t_0^+\right)$=$\lim_{t \downarrow t_0}\gamma$ and $\gamma\left( t_0^-\right)=\lim_{t \uparrow t_0}\gamma$ exist for all $t_0 \in J$, satisfying
     	\begin{itemize}
     		\item[$\bullet$] $\pi\left( \gamma\left( t_0^+ \right) \right)=\pi\left( \gamma\left( t_0^- \right) \right)$,
     		\item[$\bullet$] $\gamma\left( t_0^+ \right),\gamma\left( t_0^- \right)\in\svm$.
     	\end{itemize}
     	\item[S2 ] For each $t\in J$, $X\left( \gamma\left( t^- \right) \right)$ (resp. $X\left( \gamma\left( t^+ \right) \right)$) is the left (resp. right) derivative of $\pi(\gamma)$ at $t$.
     	\item [S3 ] Whenever $\gamma\left( t^- \right) \neq \gamma\left( t^+ \right), \; t\in J $, there is a curve in $\pi^{-1}\left( \pi\left( \gamma\left( t^+ \right) \right) \right)$ from $\gamma\left( t^- \right)$ to $\gamma\left( t^+ \right) $ along which $V$ is monotonically decreasing.
     \end{itemize}
     \end{definition}

     \begin{remark} \leavevmode
		 \begin{itemize}%
         \item Solutions are also defined for closed or semiclosed intervals. A curve $\gamma:[\alpha,\beta]\to\E$ ($\gamma:(\alpha,\beta]\to\E$, or $\gamma:[\alpha,\beta)\to\E$) is a solution of $(V,X)$ if, for any $\alpha<\alpha'<\beta'<\beta$, the restriction $\gamma|(\alpha',\beta')$ is a solution and if $\gamma$ is continuous at $\alpha$ and $\beta$ (at $\beta$, or at $\alpha$) or if there is a curve from $\gamma\left(\alpha\right)$ to $\gamma\left(\alpha^+ \right)$ and from $\gamma\left(\beta^-\right)$ to $\gamma\left(\beta \right)$ (from $\gamma\left(\beta^-\right)$ to $\gamma\left(\beta \right)$, or from $\gamma\left(\alpha\right)$ to $\gamma\left(\alpha^+ \right)$) as in property S3 above.
		 \item Note then that $\pi(\gamma)$ is continuous.
         \item The property S3 above describes the jumping process. It basically says that if a jump occurs, it happens along some fiber $\pi^{-1}(\pi(e))$. A jump is an infinitely fast transition along a fiber passing through a singular point of $S_V$.
     \end{itemize}
     \end{remark}

     \begin{definition}[\sc Jet space]
     	Let $\pi:\E\to\B$ be a fibre bundle as before. We define $J_V^k(\E,\R)$ as the space of $k-$jets of functions $V:\E\to\R$. Similarly $J_X^k(\E,T\B)$ is defined to be the space of $k-$jets of smooth maps $X:\E\to T\B$ covering $\pi$. Finally $J^k(\E)=J_V^k(\E,\R) \oplus J_X^k(\E,T\B)$ is the space of $k-$jets of constrained equations. For a given $(V,X)$, the smooth map $j^k(V,X):\E\to J^k(\E)$ assigns to each $e\in\E$ the corresponding $k-$jets of $V$ and $X$ at $e$.
     \end{definition}

     \begin{remark}
     	The elements of $J_V^k(\E,\R)$ are equivalence classes of pairs $(V,e)$, $V\in\C^\infty(\E,\R)$, $e\in\E$; where $(V,e) \sim (V',e')$ if $e=e'$ and all partial derivatives of $(V-V')$ up to order $k$ vanish at $e$. The same idea holds for $J_X^k(\E,T\B)$ and thus for $J^k(\E)$. This equivalence relation is independent of the choice of coordinates.
     \end{remark}
     
     \begin{definition}[\sc Singularity] We say that a CDE $(V,X)$ has a \emph{singularity} at $e\in\E$ if 
     \begin{enumerate}
        \item $X(e)=0$, or
     	\item $V|\pi^{-1}\left( \pi(e)\right)$ has a degenerate critical point at $e$.
     \end{enumerate}
     \end{definition}
     
     \begin{definition}[\sc The set $\Sigma^I$] \label{def:tb} Let $I=(i_1,i_2,\ldots,i_k)$ be a sequence of positive integers such that $i_1\geq i_2 \geq \cdots \geq i_k$. The set $\Sigma^I \subset J^\ell(\E)$ ($\ell\geq k$) is the set of CDEs $(V,X)$ for whose restriction $V|\pi^{-1}\left( \pi(e)\right)$ has in $e$ a critical point of Thom Boardman symbol $I$ (see appendix \ref{app:TB} for details).	
     \end{definition}

     The following statements are shown, for example, in \cite{Arnold_singularities}

     \begin{itemize}
     	\item $J^\ell(\E)$ can be stratified since the closure of $\Sigma^I$ is an algebraic subset of  $J^\ell(\E)$,
     	\item $\Sigma^I$ is a  submanifold of $J^\ell(\E)$.
     \end{itemize}

     It is useful now to state Thom's transversality theorem in the context of constrained differential equations.

     \begin{theorem}[\sc Thom's strong transversality theorem]\label{teo:trans} Let $Q\subset J^k(\E)$ be a stratified subset of codimension $p$. Then there is an open and dense subset $\mathcal O_Q\subset\C^\infty(\E,\R)\times\C^\infty(\E,T\B)$ such that for each $(V,X)\in\mathcal O_Q$, $j^k(V,X)$ is transversal to $Q$. Therefore $\left( j^k(V,X)\right)^{-1}(Q)$ is a codimension $p$ stratified subset of $\E$. 
	
     \end{theorem}

     \begin{definition} [\sc Generic CDE]\label{def:generic}    
     Let $I=(i_1,i_2,\ldots,i_k)$ be a sequence of positive integers such that $i_1\geq i_2 \geq \cdots \geq i_k$. We say that a CDE $(V,X)$ is \emph{generic} if $j^k(V,X)$ is transversal to $\Sigma ^I\subset J^\ell(\E)$, with $(\ell\geq k)$.
     \end{definition}

     In the rest of this document, the term \emph{generic} refers to definition \ref{def:generic}.

     \begin{remark} The analysis of the present document is local. Therefore, we identify the fibre bundle $\pi:\E\to\B$ with the trivial fibre bundle $\pi:\R^n\times\R^m\to\R^m$. Moreover, by definition \ref{def:generic}, let $e\in\R^n\times\R^m$ be a point such that $V|\pi^{-1}(\pi(e))$ has a degenerate critical point at $e$. Then, for $m\leq 4$, there are local coordinates such that $V$ can be written as one of the seven elementary catastrophes of table \ref{cats}. Furthermore, the local normal form of the pair $(V,X)$ can be given as a polynomial expression. 
     \end{remark}

     \begin{definition}[\sc The Singularity and Catastrophe sets] The \emph{singularity set}, also called bifurcation set, is locally defined as
     \eq{
     	B = \left\{ (x,\a) \in S_V \; | \; \det \parcs{^2 V}{x^2} = 0 \right\}.
     }

     The projection of $B$ into the parameter space $\pi(B)$ is called the \emph{catastrophe set}, and shall be denoted by $\Delta$.
     \end{definition}
     
     As can be seen from the definitions of this section, many of the topological characteristics of a generic CDE are given by the form of the potential function $V$. It is specially important to know how the critical set of $V$ is stratified. The following example is intended to give a qualitative idea of the geometric objects that one must consider.
     
     \begin{example}[\sc Strata of the Swallowtail catastrophe] \label{ex:1} Consider a CDE $(V,X)$ where the potential function $V$ is given by the swallowtail catastrophe (see table \ref{cats}). Then we have the following sets.
     \begin{center}
     \begin{tabular}{l | r}
     \hline\\
      $\Sigma^I$ & $\Sigma^I(V)=(j^k(V,X))^{-1}\left(\Sigma^I\right)$\\[2ex]
      \hline\\[1ex]
      $\Sigma^1$ & $\sv$\\[1ex]
      $\Sigma^{1,1}$ & $B$, the catastrophe set \\[1ex]
      $\Sigma^{1,1,0}$ & The set of \emph{only} fold points \\[1ex]
      $\Sigma^{1,1,1,0}$ & The set of \emph{only} cusp points \\[1ex]
      $\Sigma^{1,1,1,1}$ & The swallowtail point
      \end{tabular}  
     \end{center}
     
     The sets $\Sigma^{i}(V)$ above are formed as follows (see appendix \ref{app:TB} for the generalization) 
     \eqn{
     & \Sigma^1(V)=\left\{ (x,\alpha)\in\R^4 \, | \, D_xV=0\right\}\\
     & \Sigma^{1,1}(V)=\left\{ (x,\alpha)\in\R^4 \, | \, D_xV=D_x^2V=0\right\}\\
     & \vdots
     }
     
     The strata are manifolds of certain dimension formed by points of the same degeneracy. In our particular example we have
     \begin{eqnarray*}
         \Sigma^{1,0}(V) =\Sigma^{1}(V)\backslash \Sigma^{1,1}(V)  & &\text{Is a three dimensional manifold of regular points of $\sv$}.\\
         \Sigma^{1,1,0}(V) =\Sigma^{1,1}(V)\backslash \Sigma^{1,1,1}(V) & &\text{Is a two dimensional manifold of fold points}.\\
         \Sigma^{1,1,1,0}(V) =\Sigma^{1,1,1}(V)\backslash \Sigma^{1,1,1,1}(V) & &\text{Is a one dimensional manifold of cusp points}.\\
         & \vdots &
     \end{eqnarray*}
     
     Note that we have the inclusion $\sv\supset B\supset\Sigma^{1,1,1}\supset\Sigma^{1,1,1,1}$, which is a generic situation \cite{Arnold_singularities, GG}. The geometric features of the critical points of $V$ have an influence on $X$. Recall that $X$ maps points of the total space to tangent vectors in the base space. Therefore, besides $\sv$ being the phase space of the solutions of $(V,X)$, a generic property of $X$ is to be transversal to the projection of the bifurcation set $B$, that is to $\Delta$.
     \end{example}

     Following example \ref{ex:1}, the critical set of the codimension 3 catastrophes are stratified as shown at the end of this section in figures \ref{fig:st-stratification}, \ref{fig:hu-stratification} and \ref{fig:eu-stratification} respectively.

     \begin{definition}[\sc Topological equivalence \cite{Takens1}] \label{def:equivalence} Let $(V,X)$ and $(V',X')$ be two constrained differential equations. Let $e\in\svm$ and $e'\in S_{V',min}$. We say that $(V,X)$ at $e$ is \emph{topologically equivalent} to $(V',X')$ at $e'$ if there exists a  \emph{local homeomorphism} $h$ form a neighborhood $U$ of $e$ to a neighborhood $U'$ of $e'$, such that if $\gamma$ is a solution of $(V,X)$ in $U$, $h\circ\gamma$ is a solution of $(V',X')$ in $U'$.
     \end{definition}

     Observe that definition \ref{def:equivalence} does not require preservation of the time parametrization, only of direction.

      \subsection{Desingularization}\label{sec:desingularization}
      
      The desingularized vector field $\oX$ of a CDE $(V,X)$ is constructed in such a way that we can relate its integral curves with the solutions of $(V,X)$. An example is given in section \ref{sub:zeeman_s_nerve_impulse_model}. The general process to obtain such vector field is described in the following lines.
     \begin{lemma} [\sc Desingularization \rm{\cite{Takens1}}]  \label{lemma:des} Consider a constrained differential equation $(V,X)$ with $V$ one of the elementary catastrophes. Then the induced \emph{smooth} vector field, called \emph{the desingularized vector field} is given by
     \eq{\label{eq:vf_des}
     	\oX = \det(d\tilde\pi)(d\tilde\pi)^{-1}X(x,\tilde\pi),
     }
     where $\tilde\pi=\pi|S_V$. Furthermore, given the integral curves of the vector field $\oX$ and the map $\tilde\pi$, it is possible to obtain the solution curves of $(V,X)$. 
	
     \end{lemma}

     For a proof and details see appendix \ref{app:des}. Once the desingularized vector field \eqref{eq:vf_des} is known, the solutions of $(V,X)$ are obtained from the integral curves of $\oX$. First by changing the coordinates according to the parametrization due to $\tilde\pi$. In cases where $\det(d\tilde\pi)<0$, we reverse the direction of the solutions.\\
     \begin{remark}
         Let $(V,X)$ and $(V',X')$ be topologically equivalent CDEs. From definition \ref{def:equivalence} the homeomorphism $h$ also maps $\svm$ to $S_{V',min}$. On the other hand, it is straightforward to see that right equivalent functions have diffeomorphic critical sets. This means that we can pic and fix a representative of generic potential functions. The natural choose for low number of parameters is one of the seven elementary catastrophes. Then, our problem reduces to study the topological equivalence of CDEs $(V,X)$ and $(V,X')$, that is with the same (up to right equivalence) potential function. Denote by $\oX$ and $\overline X'$ the corresponding desingularized vector fields. It is then clear that if $\oX$ and $\overline X'$ are topologically equivalent, so are the CDEs $(V,X)$ and $(V,X')$.
     \end{remark}
     Now, let us take the notation as introduced for the catastrophes in section \ref{sec:cat}. We have the following list of desingularized vector fields. 

     \begin{corollary}\label{cor:cod3} Let $(V,X)$ be a constrained differential equation with the potential function $V$ given by a codimension 3 catastrophe (see table \ref{cats}). Let the map $X:\E\to T\B$ be given in general form as $X=f_a\parc{a}+f_b\parc{b}+f_c\parc{c}$, where $f_a, f_b, f_c$ are smooth functions of the total space $\E$. Then the corresponding desingularized vector fields  $\oX$ read as

     \begin{itemize}
     	\item Swallowtail: 
         \begin{small}
             \begin{flalign}
                 \oX = -(4x^3+2ax+b)f_a\parc{a} -(4x^3+2ax+b)f_b\parc{b}+(x^2f_a+xf_b+f_c)\parc{x}. &&
             \end{flalign}
         \end{small}
     	\item Elliptic Umbilic:
         \begin{small}
             \begin{flalign}
                 \begin{array}{@{\hspace{0mm}}r@{\;}l@{\hspace{0mm}}}
                  	 \oX = &(4a^2-36x^2-36y^2)f_a\parc{a} + \left((12x^2-4ax-12y^2)f_a+(6x-2a)f_b-6yf_c\right)\parc{x} +\\
                            &(-4y(a+6x)f_a-6yf_b-(2a+6x)f_c)\parc{y}. 
                 \end{array} &&
             \end{flalign} \hfill
         \end{small}
     	\item Hyperbolic Umbilic: 
         \begin{small}
             \begin{flalign}
                 \oX =(36xy-a^2)f_a\parc{a} + ((ax-6y^2)f_a-6y f_b + a f_c)\parc{x} + ((ay-6x^2)f_a+af_b-6xf_c)\parc{y}. &&
             \end{flalign}
         \end{small}
	
     \end{itemize}
	
     \end{corollary}
	 \begin{proof}
	 	Straightforward computations following lemma \ref{lemma:des}.
	 \end{proof}
     We end this section with Takens's theorem on normal forms of constrained differential equations with two parameters.

     \begin{theorem}[\sc Takens's Normal Forms of CDEs \textup{\cite{Takens1}}]\label{teo:takens_cde}
     Let $\pi:\E \to \B $ be as in definition \ref{def:CDE} and let $\dim(\B)=2$. Then there are {\rm{12 normal forms}} (under topological equivalence, definition \ref{def:equivalence}) of generic constrained differential equations, which are given by

     \begin{minipage}[t]{0.3\textwidth}
     \begin{tabular}[t]{l|l}
     	\hline
     	\multicolumn{2}{c}{\emph{Regular}}\\
     	\hline
     	\multicolumn{1}{c|}{$V(x,a,b)$} &  \multicolumn{1}{c}{$X(x,a,b)$}\\[0ex]\hline\\[-2ex]
     	\multirow{4}{*}{\footnotesize $\dfrac{1}{2}x^2$}	& \footnotesize $\parc{a}$ \\[2ex]\cline{2-2}\\[-2ex]
     										& \footnotesize $a\parc{a}+b\parc{b} $\\[2ex]\cline{2-2}\\[-2ex]
     										& \footnotesize $a\parc{a}-b\parc{b}$\\[2ex]\cline{2-2}\\[-2ex]
     										& \footnotesize $-a\parc{a}-b\parc{b}$								

     \end{tabular}
     \end{minipage}
     \begin{minipage}[t]{0.35\textwidth}
     \begin{tabular}[t]{l|l}
     	\hline
     	\multicolumn{2}{c}{\emph{Fold}}\\
     	\hline
     	\multicolumn{1}{c|}{$V(x,a,b)$} &  \multicolumn{1}{c}{$X(x,a,b)$}\\[0ex]\hline\\[-2ex]
     	\multirow{6}{*}{\footnotesize $\dfrac{1}{3}x^3+ax$} & \footnotesize $\parc{a}$ \\[2ex]\cline{2-2}\\[-2ex]
     										  & \footnotesize $-\parc{a}$ \\[2ex]\cline{2-2}\\[-2ex]
     										  & \footnotesize $(a+3x)\parc{a}+\parc{b}$ \\[2ex]\cline{2-2}\\[-2ex]
     										  & \footnotesize $(a-3x)\parc{a}+\parc{b}$\\[2ex]\cline{2-2}\\[-2ex]
     										& \footnotesize $-b\parc{a}+\parc{b}$\\[2ex]\cline{2-2}\\[-2ex]
     										& \footnotesize $(b+x)\parc{a}+\parc{b}$
     \end{tabular}
     \end{minipage}
     \begin{minipage}[t]{0.375\textwidth}
     	\begin{tabular}[t]{l|l}
     		\hline
     		\multicolumn{2}{c}{\emph{Cusp}}\\
     		\hline
     		\multicolumn{1}{c|}{$V(x,a,b)$} &  \multicolumn{1}{c}{$X(x,a,b)$}\\[0ex]\hline\\[-2ex]
     		\footnotesize $\dfrac{1}{4}x^4+ax^2+bx$ & \footnotesize $\parc{b}$ \\[2ex]\hline\\[-2ex]
     		\footnotesize $-\left(\dfrac{1}{4}x^4+ax^2+bx\right)$ & \footnotesize $\parc{b}$
     	\end{tabular}
     \end{minipage}
     \label{teo:Takens}
     \end{theorem}

     \begin{remark}\label{remark_cde}
      $ $
      \begin{itemize}
          \item In the fold case of theorem \ref{teo:Takens}, one extra parameter is considered (see the catastrophes list in section \ref{sec:cat}). Due to this fact, instead of having a fold singularity point at $(x,a)=(0,0)$, there is a fold line $\lbrace (x,a,b)=(0,0,b) \rbrace$. In the case $\E$ is 2-dimensional, this is, $(V,X)=\left(\sfrac{x^3}{3}+ax,g(x,a)\parc{a}\right)$, the corresponding normal forms read
        \eq{
        V(x,a)&=\sfrac{x^3}{3}+ax \quad , \qquad X=\parc{a}.\\
        V(x,a)&=\sfrac{x^3}{3}+ax \quad , \qquad X=-\parc{a}.\\
        }
  
       \item Although the classification under topological equivalence may seem too coarse, it is the simplest one. Recall the well-known fact \cite{Arnold2, HT} that there is no topological difference between the phase portraits shown in figure \ref{fig:2dvfs}.

     \begin{figure}[htbp]
     \centering 
     \subfloat { \makebox[0.25\textwidth] {\includegraphics[scale=0.65]{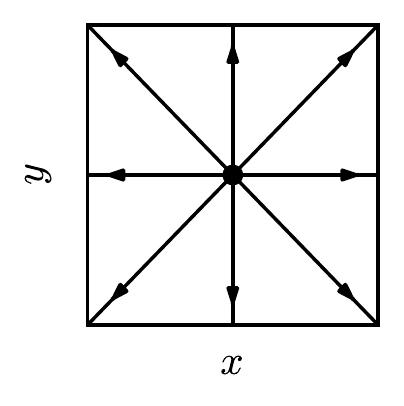}   }} 
       \subfloat { \makebox[0.25\textwidth]  {\includegraphics[scale=0.65]{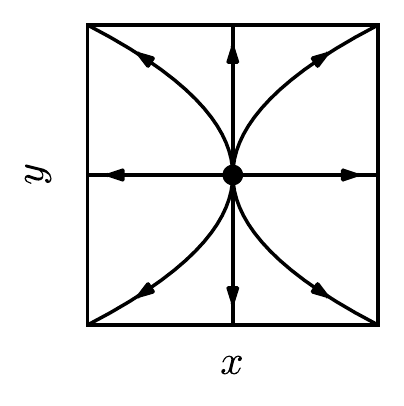} }} 
       \subfloat { \makebox[0.25\textwidth] {\includegraphics[scale=0.65]{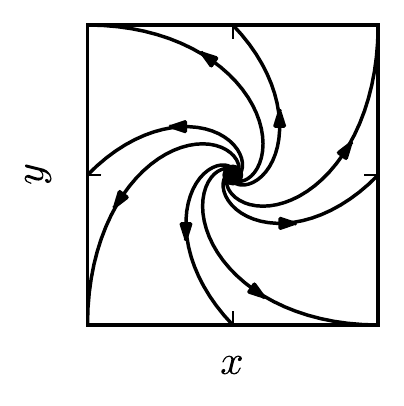}  }}
       \caption{ \label{fig:2dvfs} Topologically equivalent sources.}
     \end{figure}

     On the other hand, for application purposes, a smoother equivalence relation could be required. This would give an infinite classification since for two vector fields to be smoothly equivalent, their linear parts are to have the same spectrum. Still, if desired, the procedure to obtain a smooth normal form follows almost the same lines as below. The only difference is to skip the center manifold reduction, see section \ref{sec:nf}.
      \end{itemize}
     \end{remark}

     \begin{figure}[htbp]\centering
             \includegraphics[scale=1]{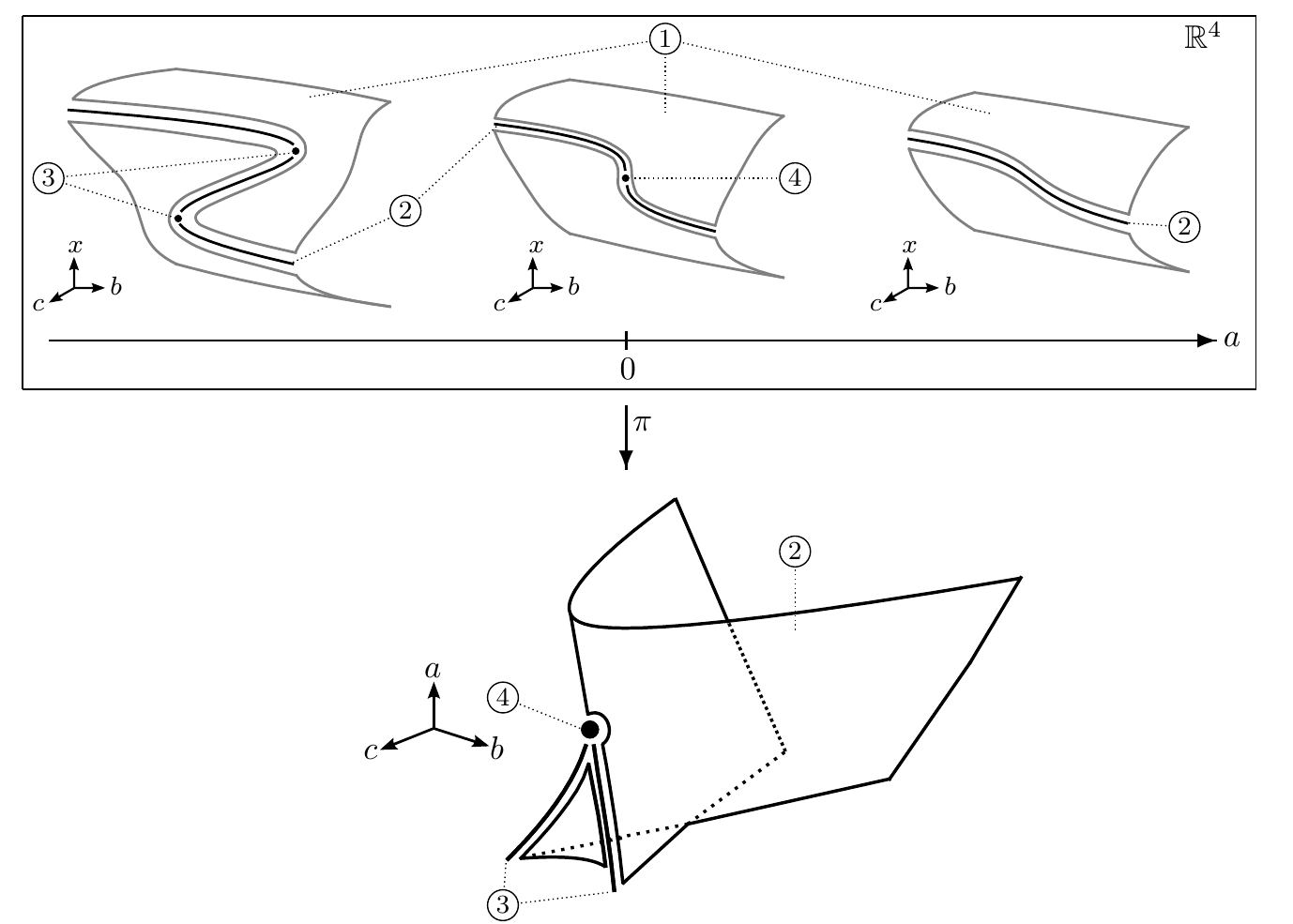} 
             \caption{Stratification of the swallowtail catastrophe. The total space is \ensuremath{\R^4}. Therefore, we show some representative tomographies. In the top figure we show the stratification of the set of critical points of the swallowtail catastrophe (refer to example \ref{ex:1}). \protect\circled{1} represents the \ensuremath{3}-dimensional set of regular points of \ensuremath{\sv}, this is \ensuremath{\sv\backslash B}.  \protect\circled{2} indicates a \ensuremath{2}-dimensional surface of folds. \protect\circled{3} denotes a $1$-dimensional curve of cusps. \protect\circled{4} represents the central singularity (at the origin) which is the swallowtail point. Note that with such notation \ensuremath{B=\text{\protect\circled{2}}\cup\text{\protect\circled{3}}\cup\text{\protect\circled{4}}}. In the bottom picture we present the projection of the singularity set, this is \ensuremath{\Delta=\pi(B)}. The same numbered notation is used to indicate the different strata.}
			 \label{fig:st-stratification}
         \end{figure}
         \newpage
     \begin{figure}[htbp]\centering    
             \subfloat[Stratification of the Hyperbolic Umbilic.]{\includegraphics[scale=1]{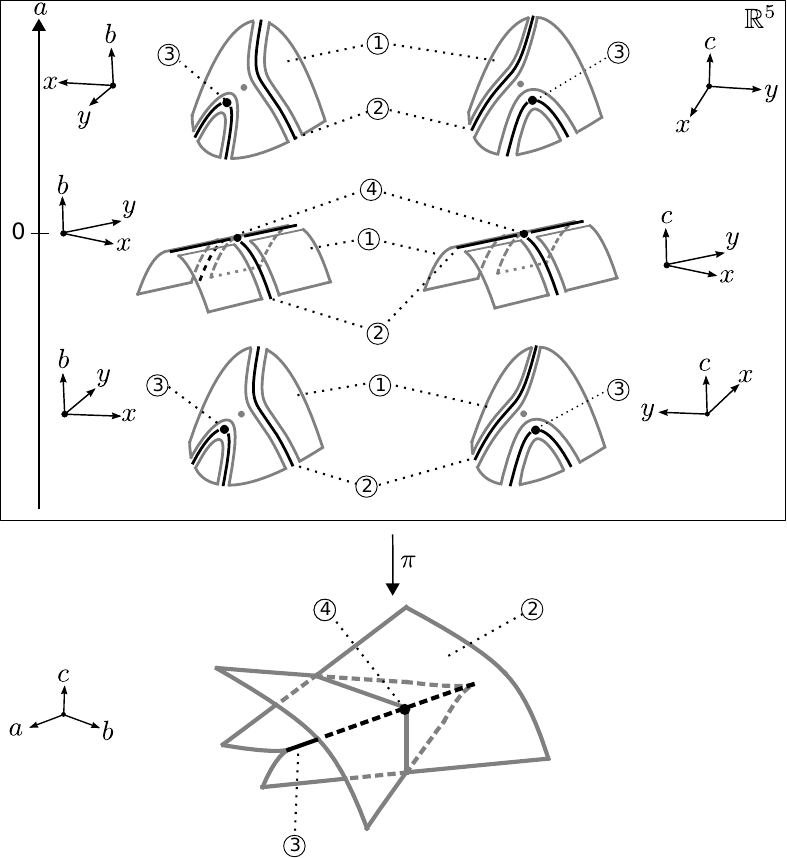} \label{fig:hu-stratification}}\hspace*{0.1cm}
             \subfloat[Stratification of the Elliptic Umbilic.]{\includegraphics[scale=.925]{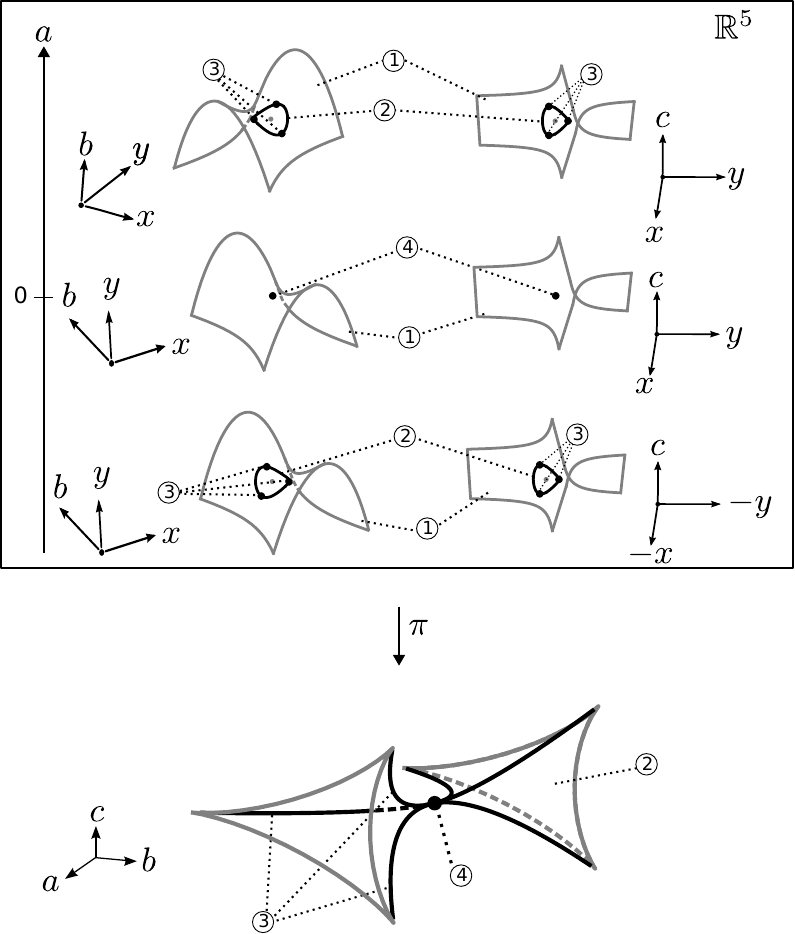} \label{fig:eu-stratification}}
     \caption{We follow the same numbered notation as in figure \protect\ref{fig:st-stratification}. \protect\circled{1} The \ensuremath{3}-dimensional manifold of regular points of \ensuremath{\sv}, this is \ensuremath{\sv\backslash B}. \protect\circled{2} \ensuremath{2}-dimensional surface of folds. \protect\circled{3} \ensuremath{1}-dimensional curve of cusps. \protect\circled{4} The central singularity corresponding to the hyperbolic umbilic in figure \protect\ref{fig:hu-stratification} and to the elliptic umbilic in figure \protect\ref{fig:eu-stratification}.}
     \end{figure}

     \begin{remark}\label{remmark_cats} Figures \ref{fig:st-stratification}, \ref{fig:hu-stratification} and \ref{fig:eu-stratification} play an important role in understanding the behavior of the solutions of generic $CDEs$ with potential function corresponding to a codimension 3 catastrophe. In each figure, the solution curves are contained in the attracting part of $\sv$. By the generic conditions of $X$, we have that for each point $p\in\Delta$, the tangent vector $X(p)$ is transverse to $\Delta$ at $p$. When a solution curve reaches a point in $B$ we generically expect to see a catastrophic change in the behavior of the solutions.     
     \end{remark}

\section{Normal forms of generic constrained differential equations with three parameters}\label{sec:nf}
    In this section we provide the main result of the present paper, phrased in theorem \ref{teo:main}. We give 16 local normal forms of generic constrained differential equations with three parameters. Thereby, we extend the existing Takens's list \cite{Takens1}. The last part of this sections contains the phase portraits of these generic CDEs.\\
	
    Due to the fact that the total space of the CDEs studied in this paper is 4 or 5 dimensional, it is worth to have a qualitative idea of what are the implication of the genericity of the map $X$. So, before stating the main result of the present document, we extend the description of codimension 3 catastrophes given by figures \ref{fig:st-stratification}, \ref{fig:hu-stratification}, and \ref{fig:eu-stratification}. We focus in describing how the geometry of $\sv$ and the genericity of $X$ relate. After this, the results stated in theorem \ref{teo:main} will seem natural.\\

    \subsection{Geometry of the codimenion 3 catastrophes.}\label{sec:geometry}

    In this section we review some of the geometrical aspects of the codimension 3 catastrophes to have an idea of what is their influence in the type of the generic desingularized vector fields. 

    \subsubsection{The Swallowtail} \label{geo:st}
    We recall that the swallowtail catastrophe is given by the potential function
    \eq{\label{eq:potential_st}
    V(x,a,b,c)=\frac{1}{5}x^5+\frac{1}{3}ax^3+\frac{1}{2}bx^2+cx.
    }
    The constraint manifold, this is the phase space of the constrained differential equation $(V,X)$ with potential function given by \eqref{eq:potential_st}, is the critical set of $V$.
    \eq{\label{eq:SV_st}
    S_V=\left\{ (x,a,b,c)\in\R^4\, | \, x^4+ax^2+bx+c=0 \right\}.
    }
    Within the constraint manifold, there are two important sets. The set $\svm$ is the attracting region of $\sv$. The set $B$ consists of singular point of $\sv$, that is where $\sv$ is tangent to the fast foliation. In the present case, the fast foliation consists of a family of curves parallel to the $x$-axis. The previous sets read
    \eq{\label{eq:SVm_st}
    \svm=\left\{ (x,a,b,c)\in \sv\, | \, 4x^3+2ax+b\geq 0 \right\},
    }
    \eq{\label{eq:B_st}
    B=\left\{ (x,a,b,c)\in \sv\, | \, 4x^3+2ax+b = 0 \right\}.
    }
    
    The projection of the singular set $B$ into the parameter space is called the catastrophe set, and it is denoted by $\Delta$ ($\Delta=\pi(B)$). As it is readily seen, the set $\sv$ is 3-Dimensional. In figure \ref{fig:geo_st} we show tomographies of $\sv$ as well as sections of $\Delta$ (see also figure \ref{fig:st-stratification} for the stratification of the swallowtail catastrophe).\\
    \newpage
    \begin{figure}[htbp]\centering
        \begin{tikzpicture}[remember picture, scale=0.9, every node/.style={scale=0.9}]
            \node(f1) at (0,0){
                \put(-145,0){\includegraphics{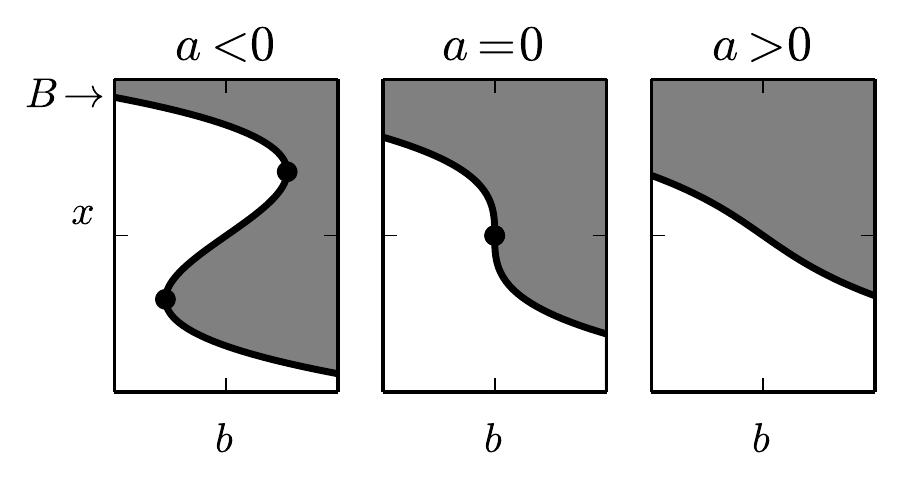}}
            };
            \node(f2) at (f1.south)[yshift=-3.5cm]{
                \put(-135,0){\includegraphics{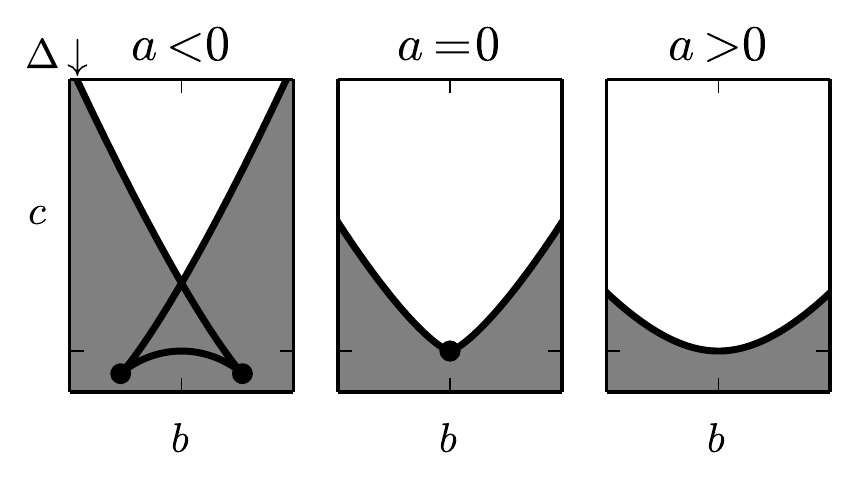}}
            };
            \node at (f1.south)[xshift=0cm, yshift=-0.5cm]{
            \begin{tikzpicture}
                \draw (0,0) edge[->] node[right] {$\pi$} (0,-1);
            \end{tikzpicture}
            };
        \end{tikzpicture}
        \caption{ From left to right we show a tomography of the \ensuremath{3}-dimensional manifold \ensuremath{\sv} for different values of \ensuremath{a} and parametrized by different coordinates. Compare with figure \protect\ref{fig:st-stratification}. The shaded region represents the stable part of \ensuremath{\sv}, that is \ensuremath{\svm}. In each figure the thick curve represents the \ensuremath{2}-dimentional set of folds. For \ensuremath{a<0} the dots stand for the \ensuremath{1}-dimensional set of cusps. For \ensuremath{a=0} the dot represents the central singularity, the swallowtail point. Note that for \ensuremath{a>0} the only singularities of \ensuremath{\sv} are fold points. The projection \ensuremath{\pi} occurs along a one dimensional fast foliation. }
        \label{fig:geo_st}
    \end{figure}
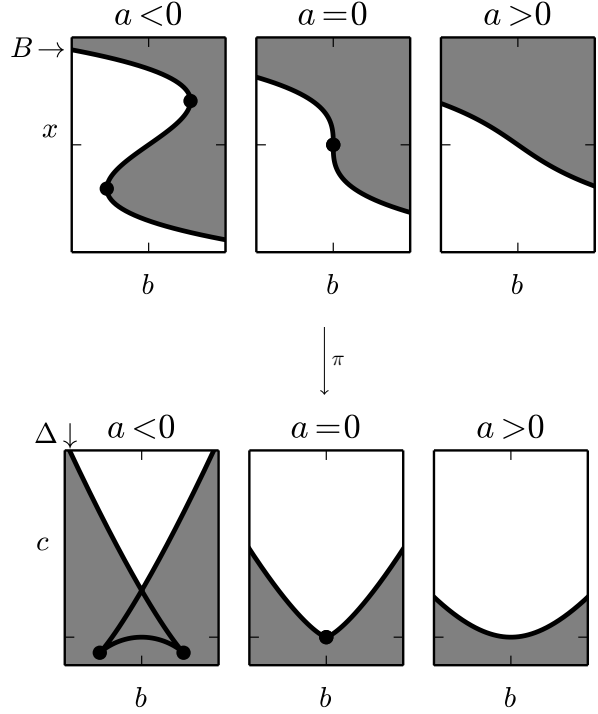
    Recall also that the desingularized vector field reads
    \eqn{
            \oX = -(4x^3+2ax+b)f_a\parc{a} -(4x^3+2ax+b)f_b\parc{b}+(x^2f_a+xf_b+f_c)\parc{x}.}
     Note that a generic condition is $\oX(0)=f_c(0)\parc{x}\neq 0$. This is, we expect that $\oX$ is given by a flow-box in a neighborhood of the central singularity. From figure \ref{fig:geo_st2} we can see that a flow-box in the direction of the $c$-axis is transversal to $\Delta$ in a neighborhood of the swallowtail point.
     
     \begin{figure}[htbp]\centering
         \includegraphics[scale=0.9]{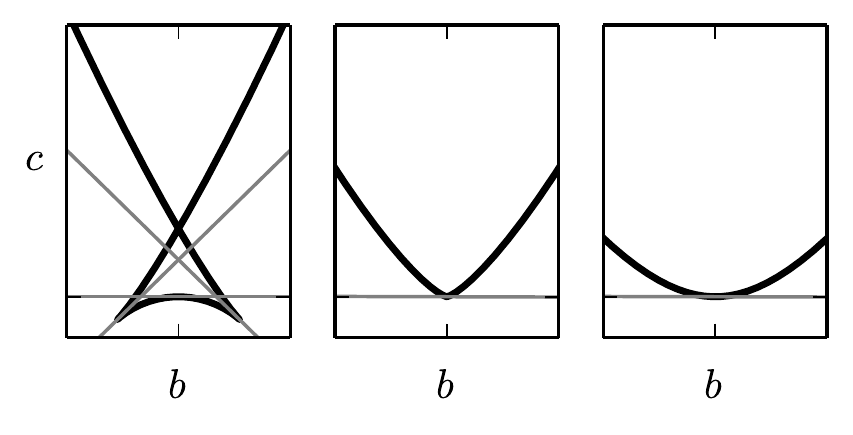}
         \caption{ The thick curve represents section of the catastrophe set \ensuremath{\Delta}. We show some tangent planes to \ensuremath{\Delta} in a neighborhood of the Swallowtail point. A generic condition of the map \ensuremath{X} is to be transversal to \ensuremath{\Delta}. So, observe that a flow-bow in the direction of the \ensuremath{c}-axis would have this property.  }
         \label{fig:geo_st2}
     \end{figure}
     
     On the other hand, the fast fibers are parallel lines to the $x$-axis. If a trajectory jumps, it does so along such a fiber. A jump of a trajectory from a singular point to a stable branches of $S_V$ is expected only when $a<0$ as this is the only case where equation defining $S_V$ \eqref{eq:SV_st} may have more than two distinct real roots. We show in figure \ref{fig:geo_st3} the projections of the singular set $B$ into the manifold $S_V$, representing the possible jumps to be encountered. 
     
     \begin{figure}[htbp]\centering
         \includegraphics[scale=0.9]{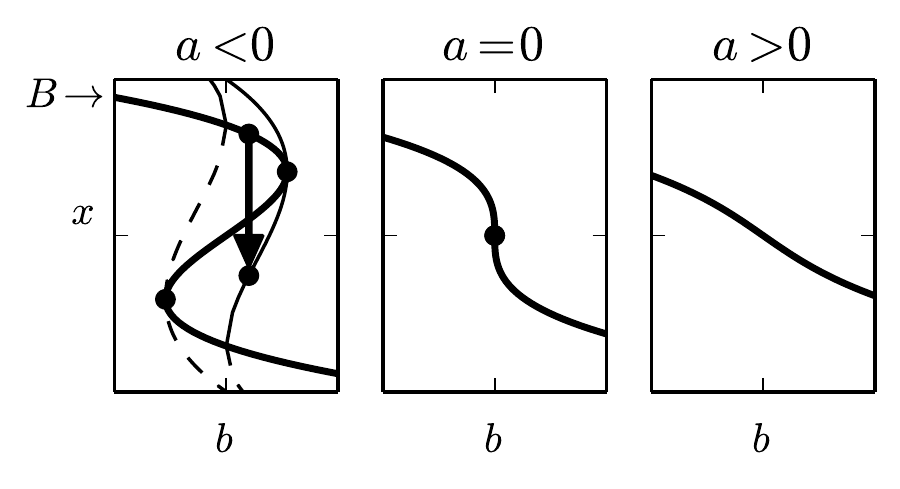}
         \caption{ For values of \ensuremath{a<0} a trajectory may jump. A jump is a infinitely fast transition from a singular point of the manifold \ensuremath{\sv} to a stable part of \ensuremath{\sv}. The transition occurs along a one dimensional fiber. The thick lines represent the singularity set \ensuremath{B}, and the thin lines represent the projection of \ensuremath{B} into \ensuremath{\sv}. Such lines represent possible arriving points when a jump occurs. We show also a possible jump situation represented as an arrow starting in \ensuremath{B} and arriving at the projection of \ensuremath{B} in to \ensuremath{\svm} (the attracting part of \ensuremath{\sv}). }
         \label{fig:geo_st3}
     \end{figure}
	
    \subsubsection{The Hyperbolic Umbilic} \label{geo:hu}
     We proceed as in the previous section with a geometric description of the hyperbolic umbilic singularity. Recall that the corresponding catastrophe reads
     \eqn{
     V(x,y,a,b,c)=x^3+y^3+axy+bx+cy.
     }
    
     Now we have two constraint variables $(x,y)$ (as opposed to the swallowtail singularity where the constraint variable is $x$). This means that the fast foliation is a family of planes parallel to $(x,y,0,0,0)\in\R^5$. The critical set of $V$ is given by
     \eqn{
     	\sv = \left\{ (x,y,a,b,c)\in\R^5 \; | \; b=-3x^2-ay, \; c=-3y^2+ax \right\}.
     }
     There are attracting points within $\sv$ defined as
     \eqn{
     	\svm = \left\{ (x,y,a,b,c)\in\sv \; | \;  \begin{bmatrix}
     		6x & a\\
     		a & 6y
             \end{bmatrix} \geq 0\right\}.
     }
    The singular set of $\sv$ is formed by all the points which are tangent to the fast fibers. Recall that now the fibration is given by parallel planes to the $(x,y,0,0,0)$ space. Such singular set reads
     \eqn{
     	B = \left\{ (x,y,a,b,c)\in\sv \; | \; 36xy-a^2 = 0\right\}.
     } 
    
     We show in figure \ref{fig:geo_hu1} some tomographies of the constraint manifold $\sv$ as well as sections of the singular set $B$.\\
    
     \newpage
     \begin{figure}[htbp]\centering
         \begin{tikzpicture}[remember picture, scale=0.9, every node/.style={scale=0.9}]
             \node(f1) at (0,0){
                 \put(-145,0){\includegraphics{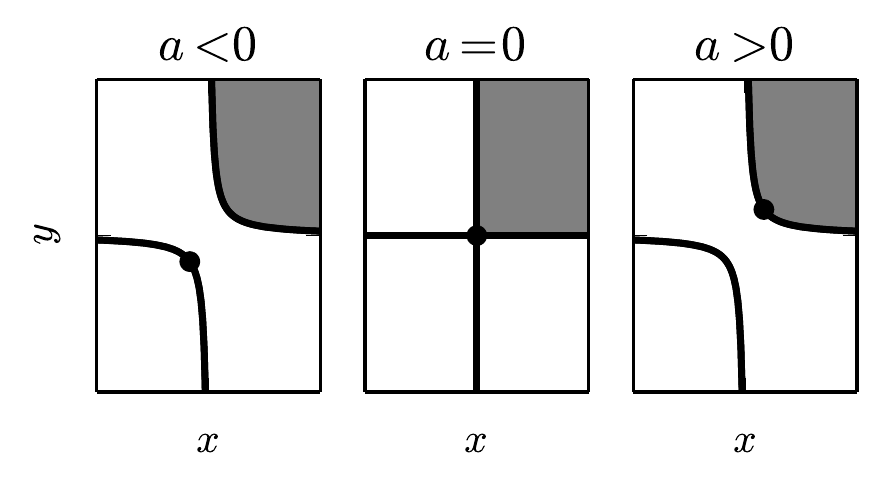}}
             };
             \node(f2) at (f1.south)[yshift=-3.5cm]{
                 \put(-135,0){\includegraphics{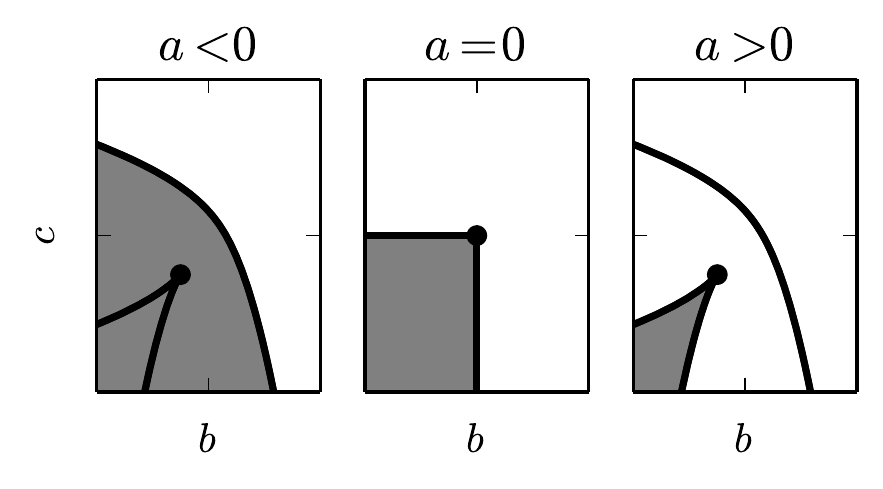}}
             };
             \node at (f1.south)[xshift=0cm, yshift=-0.5cm]{
             \begin{tikzpicture}
                 \draw (0,0) edge[->] node[right] {$\pi$} (0,-1);
             \end{tikzpicture}
             \begin{tikzpicture}[overlay]
                 \node at (-3.5,5) {$B_1$};
                 \node at (-3,2.5) {$B_2$};
                 \node at (-3.0,-1.5) {$\pi(B_1)$};
                 \node at (-4.65,-3.25) {$\pi(B_2)$};
             \end{tikzpicture}
             };
         \end{tikzpicture}
         \caption{ From left to right we show a tomography of the \ensuremath{3}-dimensional manifold \ensuremath{\sv} for different values of \ensuremath{a} and parametrized by different coordinates. Compare with figure \protect\ref{fig:hu-stratification}. The shaded region represents the stable part of \ensuremath{\sv}, that is \ensuremath{\svm}. For reference purposes, the singularity set \ensuremath{B} is divided into two components \ensuremath{B_1} and \ensuremath{B_2}. In each figure the thick curve represents the \ensuremath{2}-dimentional set of folds. For \ensuremath{a\neq 0} the dots stand for the \ensuremath{1}-dimensional set of cusps. For \ensuremath{a=0} the dot represents the central singularity, the hyperbolic umbilic point, which correspond to the intersection of the cusp lines. Recall that \ensuremath{\pi} is a projection from the total space to the parameter space, and occurs along the two dimensional fast foliation.  }
         \label{fig:geo_hu1}
     \end{figure}
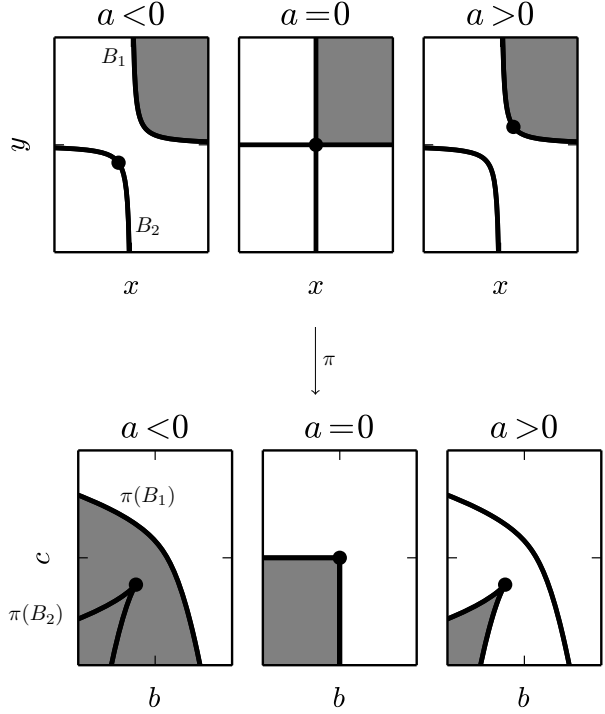
    
     Now, recall that the desingularized vector field reads
     \begin{small}
         \eqn{
         	\oX=(36xy-a^2)f_a \parc{a}+\left( \left(-6y^2+ax\right)f_a-6y f_b + a f_c \right)\parc{x}+\left( \left(-6x^2+ay\right)f_a+a f_b - 6x f_c \right)\parc{y}.
         }   
     \end{small}
    
     The vector field $\oX$ has generically an equilibrium point at the origin. It can also be shown that such point is isolated within a sufficiently small neighborhood of the origin. Therefore, in contrast with the swallowtail case, we do not expect that a generic vector field $\oX$ has the form of a flow-box. Note however, from the linearization of $\oX$ around the origin, that the hyperbolic eigenspace is two dimensional and the center eigenspace is one dimensional (see section \ref{hu} for details). So, we expect to have a 1-dimensional center manifold and two hyperbolic invariant manifolds intersecting at the origin. Such manifolds arrange the whole dynamics in a small neighborhood of the central singularity, the hyperbolic umbilic point. We expect that $\oX$ meets transversally the set $\pi(B)$. 
    
     The transversality of $X$ to $\pi(B)$ means that $\oX$ is also transversal to $B$. Such transversality property is depicted in figure \ref{fig:geo_hu2}.
    
     \begin{figure}[htbp]\centering
         \begin{tikzpicture}[remember picture, scale=0.9, every node/.style={scale=0.9}]
             \node(f1) at (0,0){
                 \put(-145,0){\includegraphics{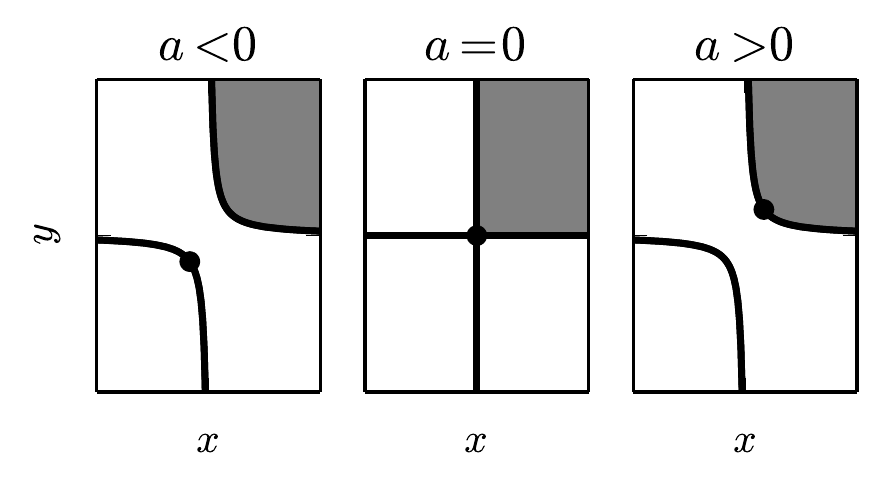}}
             };
         \end{tikzpicture}
         \begin{tikzpicture}[overlay]
             \draw (-1,3.5) edge[-]  (0,3.5);
             \draw (0.2,2) edge[-]  (0.2,2.75);
         \end{tikzpicture}
         \caption{ The transversality property of \ensuremath{\oX} with respect to \ensuremath{B} means that the integral curves of \ensuremath{\oX} are tangent to the thin lines depicted. Recall that if \ensuremath{\oX} is transversal to \ensuremath{B|(a=0)} (center picture), then \ensuremath{\oX} is also transversal to a small perturbation of \ensuremath{B|(a=0)} (left and right pictures). }
         \label{fig:geo_hu2}
     \end{figure}
    
     It is worth to take a closer look to figure \ref{fig:geo_hu1}, specially to the case $a<0$. Observe in the parameter space $(a,b,c)$ that within the shaded region $\svm$, there appear to be a set of singularities $\pi(B_2)$. However this is only a visual effect due to the projection map $\pi$. We can note from the the same picture in the space $(x,y,a)$, that the trajectories in $\svm$ cannot meet the set $B_2$.\\
      
     The jumping behavior is now more complicated. Mainly because a jump may occur along a plane parallel to the $(x,y,0,0,0)$ space. However, two important facts can be seen from figure \ref{fig:geo_hu1}. First, the set $\svm$ is one connected component. Second, as explained in the previous paragraph, we can see that there is no superposition (along the fibers) of points in $\svm$ and points in $B$ (compare with the diagram of the swallowtail given in figure \ref{fig:geo_st}). This means that along the projection $\pi$ it is not possible to join a point in $B$ with a point in $\svm$. These facts lead us to conjecture that there are not jumps for generic CDEs with a hyperbolic umbilic singularity. Such idea is proved in section \ref{sec:jumps}
	
    \subsubsection{The Elliptic Umbilic} \label{geo:eu}
	    Now we provide some insight on the geometry of the elliptic umbilic catastrophe, which is given by
	    \eqn{
	    V(x,y,a,b,c)=x^3-3xy^2+a(x^2+y^2)+bx+cy.
	    }
	    As in the hyperbolic umbilic case, the fast fibration is now two dimensional. The constraint manifold, the set of critical points of $V$ reads
	    \eqn{
	    	\sv = \left\{ (x,y,a,b,c)\in\R^5 \; | \; b=-3x^2-3y^2-2ax, \; c=-6xy-2ay \right\}.
	    }

	    As before, within $\sv$ there is a set of attracting points and a set of singular points. Each of such sets are given as
	    \eqn{
	    	\svm = \left\{ (x,y,a,b,c)\in\sv \; | \; \det\begin{bmatrix}
	    		6x+2a & 6y\\
	    		6y & 6x+2a
	            \end{bmatrix} \geq 0\right\},
	    }
	    which is equivalent to the condition $36x^2+36y^2-4a^2\geq 0$ and $a>0$. The set of singular points is given by
	    \eqn{
	    	B = \left\{ (x,y,a,b,c)\in\sv \; | \; 36x^2+36y^2-4a^2=0 \right\}.
	    }
    
	    We show in figure \ref{fig:geo_eu1} some tomographies of the constraint manifold $\sv$ as well as sections of the singular set $B$.\\
    
	    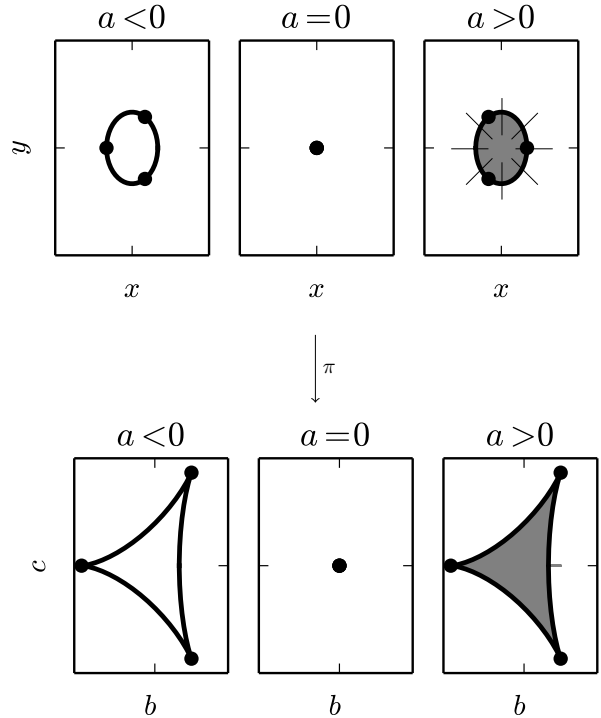
\begin{figure}[htbp]\centering
	        \begin{tikzpicture}[remember picture, scale=0.9, every node/.style={scale=0.9}]
	            \node(f1) at (0,0){
	                \put(-145,0){\includegraphics{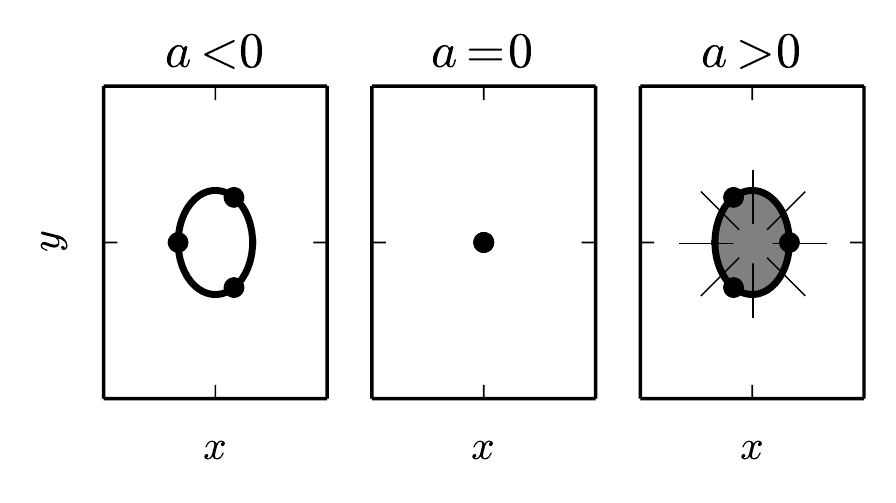}}
	            };
	            \node(f2) at (f1.south)[yshift=-3.5cm]{
	                \put(-135,0){\includegraphics{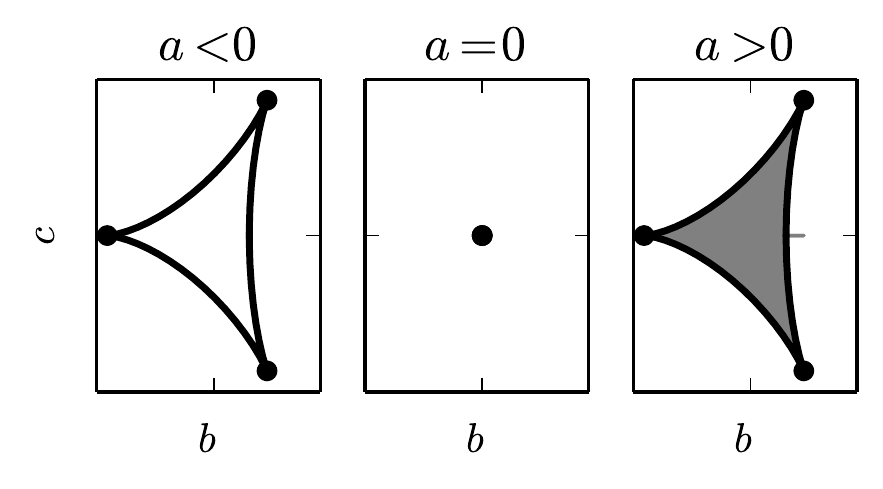}}
	            };
	            \node at (f1.south)[xshift=0cm, yshift=-0.5cm]{
	            \begin{tikzpicture}
	                \draw (0,0) edge[->] node[right] {$\pi$} (0,-1);
	            \end{tikzpicture}
	            };
	        \end{tikzpicture}
	        \caption{ From left to right we show a tomography of the \ensuremath{3}-dimensional manifold \ensuremath{\sv} for different values of \ensuremath{a} and parametrized by different coordinates. Compare with figure \protect\ref{fig:eu-stratification}. The shaded region represents the stable part of \ensuremath{\sv}, that is \ensuremath{\svm}. In each figure the thick curve represents the \ensuremath{2}-dimentional set of folds. For \ensuremath{a\neq 0} the dots stand for the \ensuremath{1}-dimensional set of cusps. For \ensuremath{a=0} the dot represents the central singularity, the hyperbolic umbilic point, which correspond to the intersection of the cusp lines. Recall that \ensuremath{\pi} is a projection from the total space to the parameter space.   }
	        \label{fig:geo_eu1}
	    \end{figure}
    
	    The desingularized vector field in this case reads
    
	    \eqn{
	    	 \oX = &(4a^2-36x^2-36y^2)f_a\parc{a} + \left((12x^2-4ax-12y^2)f_a+(6x-2a)f_b-6yf_c\right)\parc{x} + \\
	    	 		&(-4y(a+6x)f_a-6yf_b-(2a+6x)f_c)\parc{y},  
	    	}
	        and as in the Hyperbolic Umbilic case, there is generically an equilibrium point at the origin. Similar arguments as before then apply. Namely, we expect that the vector field has a 1-dimensional center manifold and two hyperbolic invariant manifold intersecting at the origin. A qualitative picture of the transversality of $\oX$ with respect to $B$ is shown in figure \ref{fig:geo_eu2} 
        
	        \begin{figure}[htbp]\centering
				\includegraphics{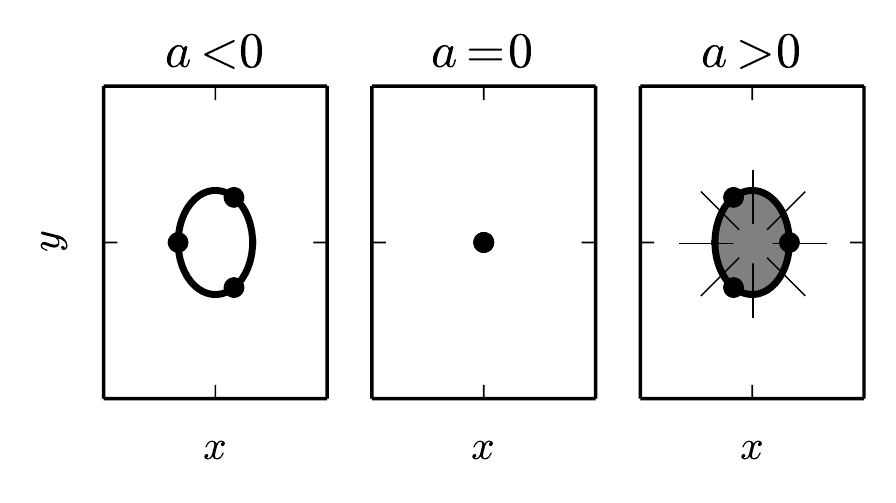}
	            \caption{ The transversality property of \ensuremath{\oX} with respect to \ensuremath{B} means that the integral curves of \ensuremath{\oX} are tangent to the thin lines depicted in the right picture.  }
	            \label{fig:geo_eu2}
	        \end{figure}
        
	        Regarding the jumps, the same arguments as for the hyperbolic umbilic catastrophe apply. Observe from figure \ref{fig:geo_eu1} that it is not possible to join points in $B$ with points in $\svm$ along the fibers.
	
    \newpage
    \subsection{Main theorem}\label{sec:main_theorem}
	    In this section we provide a list of generic CDEs with three parameters. In contrast with Takens's list of normal forms \cite{Takens1}, the result in this sections includes CDEs with two dimensional fast fibers. As it was mentioned in section \ref{sec:cde} folds and cusps (lower codimension singularities) also appear as generic singularities of CDEs with three parameters. However the qualitative behavior in the neighborhood the solutions near folds and cusps can be understood from Takens's list \cite{Takens1}. The novelty of theorem \ref{teo:main} is the description of the solutions of CDEs in a neighborhood of a swallowtail, hyperbolic, and elliptic umbilic singularity.

	\begin{theorem}\label{teo:main} Let $(V,X)$ be a generic constrained differential equation with three parameters. Then $(V,X)$ is topologically equivalent to one of the following $16$ polynomial local normal forms.
	\end{theorem}

	\renewcommand{\arraystretch}{2}
	\begin{small}
	\begin{tabular}[!htbp]{|l|l|r|}
	    	\hline
	    	\multicolumn{3}{|c|}{Regular}\\
	    	\hline
	    	\multicolumn{1}{|c|}{$V(x,a,b,c)$} &  \multicolumn{1}{c|}{$X(x,a,b,c)$} & \multicolumn{1}{c|} {Type}\\ [0ex]\hline
	    	\multirow{5}{*}{\footnotesize $\dfrac{1}{2}x^2$}	& \footnotesize $\parc{a}$ & Flow-box\\[1ex]
	        \cline{2-3}
	    										& \footnotesize $a\parc{a}+b\parc{b}+c\parc{c} $ & Source\\[1ex]
	                                            \cline{2-3}
	    										& \footnotesize $a\parc{a}+b\parc{b}-c\parc{c}$ & Saddle-1\\[1ex]
	                                            \cline{2-3}
	    										& \footnotesize $a\parc{a}-b\parc{b}-c\parc{c}$ & Saddle-2\\[1ex]
	                                            \cline{2-3}
	    										& \footnotesize $-a\parc{a}-b\parc{b}-c\parc{c}$ & Sink\\\hline
	\end{tabular}
	\end{small}
    
	    \vspace*{0.5cm}
    
	\begin{small}
	        \begin{tabular}[t]{|l|l|r|}
	        	\hline
	        	\multicolumn{3}{|c|}{\emph{Fold}}\\
	        	\hline
	        	\multicolumn{1}{|c|}{$V(x,a,b,c)$} &  \multicolumn{1}{c|}{$X(x,a,b,c); \quad (\rho=\pm 1, \, \delta\in\R)$}  & \multicolumn{1}{c|} {Type} \\
	        	\hline
	        	\multirow{5}{*}{\footnotesize $\dfrac{1}{3}x^3+ax$} & \footnotesize $\parc{a}$ & Flow-box-1 \\[1ex]\cline{2-3}
	        										  & \footnotesize $-\parc{a}$ & Flow-box-2 \\[1ex]\cline{2-3}
	        										  & \footnotesize $\left( 3x+\frac{1}{2}b+\frac{1}{2}c \right)\parc{a}+(c-b)^2 \left( \rho +\delta(c-b) \right) \left(-\parc{b}+\parc{c}\right)+\frac{1}{2}\left( \parc{b}+\parc{c}\right) $ & Source\\[1ex]\cline{2-3}
	        										& \footnotesize $\left( -3x+\frac{1}{2}b+\frac{1}{2}c \right)\parc{a}+(c-b)^2 \left( \rho +\delta(c-b) \right) \left(-\parc{b}+\parc{c}\right)+\frac{1}{2}\left( \parc{b}+\parc{c}\right)$ & Sink\\[1ex]\cline{2-3}
	        										& \footnotesize $-\left(\frac{1}{2}b+\frac{1}{2}c \right)\parc{a}+(c-b)^2 \left( \rho +\delta(c-b) \right) \left(-\parc{b}+\parc{c}\right)+\frac{1}{2}\left( \parc{b}+\parc{c}\right)$ & Saddle	\\\hline
	        \end{tabular}
	        \begin{remark}
	            If $b=c$, these fold normal forms reduce to those of theorem \ref{teo:takens_cde}. 
	        \end{remark}
            
	\clearpage

	\begin{tabular}[t]{|l|l|r|}
		\hline
		\multicolumn{3}{|c|}{\emph{Cusp}}\\
		\hline
		\multicolumn{1}{|c|}{$V(x,a,b,c)$} &  \multicolumn{1}{c|}{$X(x,a,b,c)$} & \multicolumn{1}{c|} {Type}\\
		\hline
		\footnotesize $\dfrac{1}{4}x^4+ax^2+bx$ & \footnotesize $\parc{b}$ & Flow-box \\[2ex]\hline
		\footnotesize $-\left(\dfrac{1}{4}x^4+ax^2+bx\right)$ & \footnotesize $\parc{b}$ & (Dual) Flow-box\\\hline
	\end{tabular}
	\begin{tabular}[t]{|l|l|r|}
		\hline
		\multicolumn{3}{|c|}{\emph{Swallowtail}}\\\hline
		\multicolumn{1}{|c|}{$V(x,a,b,c)$} &  \multicolumn{1}{c|}{$X(x,a,b,c)$} & \multicolumn{1}{c|} {Type}\\[0ex]\hline
		\footnotesize $\dfrac{1}{5}x^5+\dfrac{1}{3}ax^3+\dfrac{1}{2}bx^2+cx$ & \footnotesize $\parc{c}$ & Flow-box\\\hline
	\end{tabular}

	\vspace*{0.5cm}

	        \begin{tabular}{|l|l|r|}
	        	\hline
	        	\multicolumn{3}{|c|}{\emph{Hyperbolic Umbilic}}\\
	        	\hline
	        	\multicolumn{1}{|c|}{$V(x,y,a,b,c)$} &  \multicolumn{1}{c|}{$X(x,y,a,b,c)$} & \multicolumn{1}{c|}{Type}\\
	            \hline
	            \multirow{4}{*}{\footnotesize $x^3+y^3+axy+bx+cy$} & \footnotesize $6\Phi(a)\parc{a}-\left(\Phi(a)(6x+6y-a)-6xy+\frac{a^2}{6}\right)\left( \parc{b}+\parc{c} \right)$ & Center-Saddle\\ [2ex]\cline{2-3}
	            & \footnotesize $6\sum_{\ell=2}^k\sum_{j=0}^{2j=\ell}\rho_{\ell,j}A_{\ell,j} \parc{a}+\left( \frac{a^2}{6}-6xy\right)\parc{b}+\left( -\frac{a^2}{6}-6xy\right)\parc{c}+$ & Center\\[2ex]
	                        & \footnotesize $\sum_{\ell=2}^k\left( (6x+a-6y)\sum_{j=0}^{2j=\ell}\rho_{\ell,j}A_{\ell,j}+\sum_{j=0}^{2j+1=\ell}\left(\frac{a}{6}\right)^{-1}A_{\ell,j}B_{\ell,j} \right)\parc{b}+$ & \\[2ex]
	                        &\footnotesize $\sum_{\ell=2}^k\left( (6y+a-6x)\sum_{j=0}^{2j=\ell}\rho_{\ell,j}A_{\ell,j}+\sum_{j=0}^{2j+1=\ell}\left(\frac{a}{6}\right)^{-1}A_{\ell,j}\overline B_{\ell,j} \right)\parc{c}$ & \\\hline
	        \end{tabular}

	\vspace*{0.25cm}
	Where 
	\begin{footnotesize}
	    \eqn{
	    \Phi(a) &= \pm\frac{a^2}{36}+\frac{\delta a^3}{216}, \qquad \delta\in\R\\
	    A_{\ell,j} &=\left(\frac{a}{6}\right)^{\ell-j}\Delta^j,  \qquad \Delta=\left(\frac{a}{108}\right)\left(a^2+18(x^2+y^2) +6(ax+ay) \right)\\
	    B_{\ell,j}&=-6xC_{\ell,j}-a\overline C_{\ell,j} \\
	    \overline B_{\ell,j}&=-aC_{\ell,j}-6y\overline C_{\ell,j} \\
	    C_{\ell,j}&=\eta_{\ell,j}\left( \frac{a}{6}+x \right)+\sigma_{\ell,j}\left( \frac{a}{6}+y \right)  \\
	    \overline C_{\ell,j} &= \eta_{\ell,j}\left( \frac{a}{6}+y \right)-\sigma_{\ell,j}\left( \frac{a}{6}+x \right),
	    }    
	\end{footnotesize}

	with $\rho_{\ell,j},\eta_{\ell,j},\sigma_{\ell,j}\in\R$.

	        \vspace*{0.25cm}
	        \begin{tabular}{|l|l|r|}
	        	\hline
	        	\multicolumn{3}{|c|}{ \emph{Elliptic Umbilic}}\\
	        	\hline
	        	\multicolumn{1}{|c|}{$V(x,y,a,b,c)$} &  \multicolumn{1}{c|}{$X(x,y,a,b,c)$} & \multicolumn{1}{c|}{Type}\\
	        	\hline
	        	\multirow{1}{*}{\footnotesize $x^3-3xy^2+a(x^2+y^2)+bx+cy$} & \footnotesize $A\parc{a}+\frac{B}{\sqrt{2}}\left( \parc{b}+\parc{c} \right) -\frac{1}{\sqrt{2}}\left(2xA\parc{b}+2yA\parc{c} \right)$ & Center-Saddle\\\hline
	        \end{tabular}\\

	        Where $A =  \frac{1}{9}\left(\pm 3a^2 +\delta a^3\right), \, \delta\in\R$, and $B = -6x^2-6y^2+\frac{2}{3}a^2$.\\

	\end{small}

	We show in section \ref{sec:pics} some phase portraits of the CDEs of theorem \ref{teo:main}. Recall remark \ref{remmark_cats} for the relationship between the list of normal forms and figures  \ref{fig:st-stratification}, \ref{fig:hu-stratification} and \ref{fig:eu-stratification}.
	
    \newpage
    \subsection{Proof of the main result} \label{sec:proof}
	\renewcommand{\arraystretch}{1}
    In this section we prove theorem \ref{teo:main}. We only detail the hyperbolic umbilic case as it is the most interesting one. All the other cases follow exactly the same lines. The procedure is summarized as follows.
    \begin{enumerate}
      \item Desingularization of $(V,X)$. With this we obtain the desingularized vector field $\oX$. Then we are able to use standard techniques of dynamical systems theory to obtain a polynomial normal form of $\oX$ following the next two steps.
      \item Reduction to a center manifold, see appendix \ref{app:cm}. This reduction greatly simplifies the expressions of the normal forms. 
      \item Apply Takens's normal form theorem, see appendix \ref{app:nf}. 
      \item At this stage, we have a polynomial local normal form of the vector field $\oX$. Now, recall that the form of $\oX$ is obtained by following the desingularization process described in section \ref{sec:desingularization}. So, the last step in order to write the local normal forms of a constrained differential equation $(V,X)$ is to carry out the inverse coordinate transformation performed when obtaining $\oX$.
    \end{enumerate}

    \subsubsection*{The Hyperbolic Umbilic}\label{hu}
	Following table \ref{cats}, we deal with the constrained differential equation
	\eq{ \label{eq:hu-cde}
		V(x,y,a,b,c) &= x^3+y^3+axy+bx+cy\\
		X(x,y,a,b,c) &= f_a\parc{a}+f_b\parc{b}+f_c\parc{c},
	}

	The functions $f_i(x,y,a,b,c):\R^5\to\R$, for $i=a,b,c$, are considered to be $\C^\infty$ with the generic condition $f_i(0)\neq 0$ . The constraint manifold is the critical set of the potential function $V$
	\eq{\label{eq:hu-sv}
		\sv = \left\{ (x,y,a,b,c)\in\R^5 \; | \; b=-3x^2-ay, \; c=-3y^2+ax \right\}.
	}

	The attracting region of $S_V$ is
	\eq{\label{eq:hu-svm}
		\svm = \left\{ (x,y,a,b,c)\in\sv \; | \;  \begin{bmatrix}
			6x & a\\
			a & 6y
	        \end{bmatrix} \geq 0\right\},
	}

	which is equivalent to the conditions $36xy-a^2\geq 0$ and $x+y\geq 0$. Consequently, the catastrophe set reads
	\eq{\label{eq:hu-bif}
		B = \left\{ (x,y,a,b,c)\in\sv \; | \; \det\begin{bmatrix}
			6x & a\\
			a & 6y
		\end{bmatrix} =0\right\}.
	}

	Refer to figure \ref{fig:hu-stratification} for the pictures of $\sv$ and $B$. Following the desingularization process, we choose coordinates in $S_V$. The projection into the parameter space restricted to $S_V$ is
	\eq{\tilde \pi=(a,-3x^2-ay,-3y^2-ax).}

	Observe that $\det(D\tilde\pi)\geq 0$ for points in $\svm$. By following corollary \ref{cor:cod3}, the corresponding desingularized vector field is
	\begin{small}
	    \eq{\label{eq:hu-des}
	    	\oX=(36xy-a^2)f_a \parc{a}+\left( \left(-6y^2+ax\right)f_a-6y f_b + a f_c \right)\parc{x}+\left( \left(-6x^2+ay\right)f_a+a f_b - 6x f_c \right)\parc{y}.
	    }   
	\end{small}

	The vector field $\oX$ has an equilibrium point at the origin. The corresponding linearization shows the spectrum $\left\{ 0,+6\sqrt{f_b(0)f_c(0)},-6\sqrt{f_b(0)f_c(0)} \right\}$. Considering the generic conditions on $f_b$ and $f_c$, and by referring to the center manifold theorem \ref{teo:cm}, we study the cases where $\oX$ is topologically equivalent to

	\begin{enumerate}
	    \item $\oX' (u,v,w) = f_u(u)\parc{u}+v \parc{v}- w \parc{w}, \qquad \text{or}$
	    \item $\oX' (u,v,w) = f_u(u,v,w)\parc{u} + \left(v+f_w(u,v,w)\right) \parc{w}+\left(-w+f_v(u,v,w) \right)\parc{v}$,
	\end{enumerate}

	where $f_i(0)=Df_i(0)=0$ for $i=u,v,w$. We study each case separately.

	\begin{enumerate}
	    \item Here we consider that the spectrum of $\oX$ is of the form $\left\{ 0,\lambda_1,\lambda_2 \right\}$, $\lambda_1>0>\lambda_2$, so we call it \emph{the center-saddle case}. There exists a 1-dimensional center manifold passing through the origin. Following theorem \ref{teo:nf} and noting that

	\eq{
		\left[u^2\parc{u}, u^{k-1}\parc{u}  \right]=(k-3)u^k,
	}
	we have that the $k-$jet of $\oX'$ is smoothly equivalent to 

	\eq{
	\left(\delta_1u^2+\delta_2u^3\right)\parc{u}+v\parc{v}-w\parc{w}
	} 

	for all $k\geq 3$, where $\delta_1\in\R\backslash\left\{ 0\right\}$, and $\delta_2\in\R$. With this we can further say that $\oX$ is topologically equivalent to

	\eq{\label{eq:c-s}
		\oX'=\left(\pm u^2+\delta u^3\right)\parc{u}+v\parc{v}- w \parc{w}, \qquad \delta\in\R.
	}
	Observe that $u$ is the center direction and $v,w$ are the hyperbolic (saddle) directions. Locally, the direction of the center manifold depends on the $\pm$ sign in front of the $u^2$ term of the normal form \eqref{eq:c-s}.
	\item Now we deal with a $3$-dimensional center manifold. The vector field $\oX'$ has spectrum $\lbrace 0,\lambda\imath,-\lambda\imath \rbrace, \; \lambda\in\R$, so we call it \emph{the center case}. It is convenient to introduce complex coordinates 
	\eq{
		z&=u+\imath w,\\
		\oz&=u-\imath w.
	}
	In these coordinates we have that the $1-jet$ of $\oX'$ is
	\eq{
		\oX'_1(u,z,\oz)=\imath\left( z\parc{z}-\oz\parc{\oz} \right).
	}

	Following the normal form theorem \ref{teo:nf}, we write the elements of $\mathcal H^k\otimes\mathbb{C}$ as a combination of the monomials $u^{m_1}z^{m_2}\oz^{m_3}$, where $m_1+m_2+m_3=k$, having the relations

	\eq{
		\left[ \oX'_1,\; u^{m_1}z^{m_2}\oz^{m_3}\parc{u}\right]&=\imath u^{m_1}z^{m_2}\oz^{m_3}(m_2-m_3)\parc{u},\\
		\left[ \oX'_1,\; u^{m_1}z^{m_2}\oz^{m_3}\parc{z}\right]&=\imath u^{m_1}z^{m_2}\oz^{m_3}(m_2-m_3-1)\parc{z},\\
		\left[ \oX'_1,\; u^{m_1}z^{m_2}\oz^{m_3}\parc{\oz}\right]&=\imath u^{m_1}z^{m_2}\oz^{m_3}(m_2-m_3+1)\parc{\oz}.
	}
	We can choose as a complement of the image of $\left[\oX'_1,- \right]_k$ the space spanned by

	\begin{small}
	    \eq{
	    	&\left\{u^{k-2m}z^m\oz^m\parc{u}\right\}_{m=0}^{m=k/2}\bigcup \\
	    	&\left\{ u^{k-1-2m}z^m\oz^m\left( z\parc{z}+\oz\parc{\oz}\right), \;\imath u^{k-1-2m}z^m\oz^m\left( z\parc{z}-\oz\parc{\oz}\right)  \right\}_{m=0}^{m=\frac{k-1}{2}}.
	    }
	\end{small}    

	This base is chosen so that we can easily write the normal form in the original coordinates by identifying $\left( z\parc{z}+\oz\parc{\oz}\right), \text{ and }\imath\left( z\parc{z}-\oz\parc{\oz}\right)$ with $\left(v\parc{v}+w\parc{w}\right), \text{ and }\left(v\parc{w}-w\parc{v}\right) $ respectively. Then, we have that the $k-$th order polynomial normal form of $\oX'$ reads

	\begin{footnotesize}
	    \eq{\label{eq:c}
	    	\oX'=\oX'_1+&\sum_{\ell=2}^k\left(\sum_{j=0}^{2j=\ell}\rho_{\ell j}u^{\ell-2j}(v^2+w^2)^j\parc{u}  +\right.\\
	    	&\left.\sum_{j=0}^{2j+1=l}u^{\ell-1-2j}(v^2+w^2)^j \left( \eta_{\ell j}\left( v\parc{v}+w\parc{w}\right)+\sigma_{\ell j}\left( v\parc{w}-w\parc{v}\right)\right)  \right)  ,
	    }    
	\end{footnotesize}

	where $\rho_{\ell j}, \eta_{\ell j}$, and $\sigma_{\ell j}$ are some nonzero constants. Compare with \cite{Takens2}, where the case of a vector field having eigenvalues of its Jacobian equal to $\left\{\alpha,\pm \imath \right\}, \; \alpha\neq 0$ is studied. 
	\end{enumerate}

	At this point then, we have two normal forms of the vector field $\oX'$ depending on the eigenvalues of $D_0\oX$. Recall that the solutions of $(V,X)$ are related to the integral curves of $\oX$ and therefore also to the integral curves of $\oX'$. In order to locally identify the coordinates in which we expressed $\oX'$ with the original coordinates $(x,y,a,b,c)$, we perform a linear change of coordinates such that $D_0\oX=D_0\oX'$. This linear transformation is given by
	\eq{
		\begin{bmatrix}
			a\\
			x\\
			y
		\end{bmatrix}=\begin{bmatrix}
			6 & 0 & 0\\
			1 & -1 & 1\\
			1 & 1 & 1
		\end{bmatrix}\begin{bmatrix}
			u\\
			v\\
			w
		\end{bmatrix}
	}

	in the case of the center-saddle vector field \eqref{eq:c-s}, and
	\eq{
		\begin{bmatrix}
			a\\
			x\\
			y
		\end{bmatrix}=\begin{bmatrix}
			6 & 0 & 0\\
			-1 & 0 & 1\\
			-1 & 1 & 0
		\end{bmatrix}\begin{bmatrix}
			u\\
			v\\
			w
		\end{bmatrix}
	}

	in the case of the vector field \eqref{eq:c}. By carrying out the computations, $\oX$ has respectively the $k-$th order local normal form

	\begin{enumerate}
	    \item Center-saddle case
	    \begin{small}
	        \begin{flalign}\label{eq:hu1}
	            \oX=\left( \pm a^2 +\delta a^3\right)\parc{a}+\dfrac{1}{6}\left(\left( \pm a^2 +\delta a^3\right)+a-6y\right)\parc{x}+\dfrac{1}{6}\left( \left( \pm a^2 +\delta a^3\right)+a-6x \right)\parc{y},&&
	        \end{flalign}
	    \end{small}
	    where $\delta\in\R$.
	    \item Center case
	    \begin{small}
	        \begin{flalign}\label{eq:hu2}
	            \begin{array}{@{\hspace{0mm}}r@{\;}l@{\hspace{0mm}}}
	                \oX = &\left(\frac{1}{6}a+y\right)\parc{x}-\left(\frac{1}{6}a+x\right)\parc{y}+6\sum_{\ell=2}^k\sum_{j=0}^{2j=\ell}\rho_{\ell j}\left( \frac{a}{6}\right)^{\ell-j}\Delta^j \parc{a}+\\
	                &\left( -\sum_{\ell=2}^k\sum_{j=0}^{2j=\ell}\rho_{\ell j}\left( \frac{a}{6}\right)^{\ell-j}\Delta^j+\sum_{\ell=2}^k\sum_{j=0}^{2j+1=\ell}\left( \frac{a}{6}\right)^{\ell-1-j}\Delta^jA_{\ell,j} \right)\parc{x}+\\
	                &\left( -\sum_{\ell=2}^k\sum_{j=0}^{2j=\ell}\rho_{\ell j}\left( \frac{a}{6}\right)^{\ell-j}\Delta^j+\sum_{\ell=2}^k\sum_{j=0}^{2j+1=\ell}\left( \frac{a}{6}\right)^{\ell-1-j}\Delta^j\overline A_{\ell,j} \right)\parc{y}
	            \end{array}&&
	        \end{flalign}
	    \end{small}
	where 
	\eqn{
	\Delta & =\frac{a}{108}\left( a^2 + 6ax + 6ay + 18x^2 + 18y^2 \right)\\
	A_{\ell,j} & =\eta_{\ell,j}\left( \frac{a}{6}+x \right)+\sigma_{\ell,j}\left( \frac{a}{6}+y \right)\\
	\overline A_{\ell,j} & =\eta_{\ell,j}\left( \frac{a}{6}+y \right)-\sigma_{\ell,j}\left( \frac{a}{6}+x \right), \qquad \eta_{\ell,j},\sigma_{\ell,j}\in\R\\
	}
	\end{enumerate}

	The phase portraits of \eqref{eq:hu1} and \eqref{eq:hu2} are shown in figures \ref{fig:hu-cs} and \ref{fig:hu-cc} respectively.\\

	Finally, by following lemma \ref{lemma:des} we can obtain the form of $(V,X)$. Recall that the desingularized vector field is defined by $\oX=\det(D\tilde\pi)(D\tilde\pi)^{-1}X$. This means that in principle, once we know $\oX$, $X$ is obtained as $X=\frac{1}{\det(D\tilde\pi)}D\tilde\pi\oX$. Clearly, the map $X$ is not define for points at the bifurcation set. Away from such set, $X$ is equivalent to the smooth map $D\tilde\pi\oX$. Furthermore, since $\det(D\tilde\pi)>0$ in $\svm$, the solution curves of $(V,X)$ are obtained from the integral curves of $\oX$ and by the reparametrization
	\eqn{
	b=-3x^2-ay, \qquad c=-3y^2-ax.
	}
	Straightforward computations show that the CDE $(V,X=D\tilde\pi\oX)$ with a hyperbolic umbilic singularity has the local normal forms as stated in theorem \ref{teo:main}.

    \newpage

    \subsection{Phase portraits of generic CDEs with three parameters}\label{sec:pics}
    In this section we present the phase portraits of some of the normal forms of theorem \ref{teo:main}. Recall that $\sv$ is the phase space, this is, the solution curves belong to the manifold $\sv$. Such manifolds are as depicted in figures  \ref{fig:st-stratification}, \ref{fig:hu-stratification} and \ref{fig:eu-stratification}. At the bifurcation sets $B$, the solution curves have a sudden change of behavior. It is said, a catastrophe occurs.\\

    In some words, a generic constrained differential equation with three parameters is likely to qualitatively behave as one of the pictures presented in this section.
	\subsubsection*{Regular}
	In this case the constraint manifold $S_V$ has no singularities. So the constraint manifold $\sv$ is the whole $\R^3$. In figures \ref{fig:reg-constant} and \ref{fig:reg-source} we show the phase portraits of the flow-box and source case. The pictures of the saddle-1, saddle-2 and sink are similar to figure \ref{fig:reg-source} just changing accordingly the directions of the invariant manifolds.

	    \begin{figure}[!htbp]
	        \centering
	        \subfloat[Flow-box phase portrait]{\makebox[0.35\textwidth] {\includegraphics[scale=0.8]{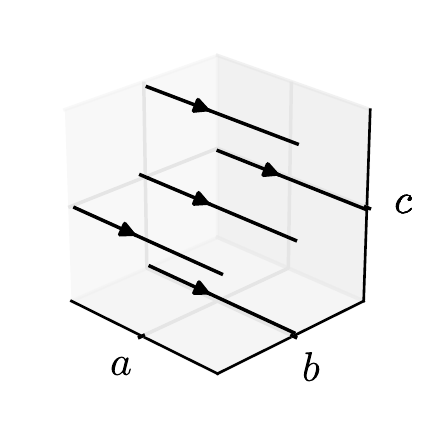}\label{fig:reg-constant}}}
	        \subfloat[Source phase portrait]{\makebox[0.35\textwidth] { \includegraphics[scale=0.8]{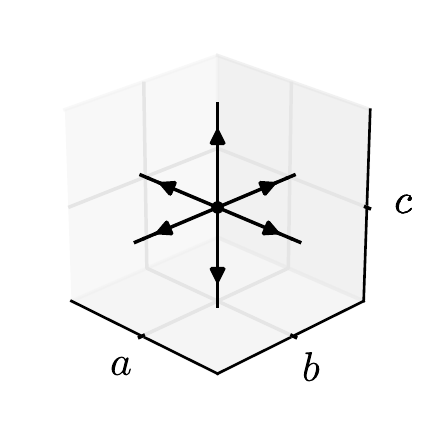}\label{fig:reg-source} }}
	        \caption{Phase portraits corresponding to the regular case. We show only two examples corresponding to the flow-box (left) and the source (right) case. As the constraint manifold \ensuremath{\sv} is regular, the only singularities that may happen are equilibrium points, this is \ensuremath{X(0)=0}. Due to the same reason, there are not jumps. The remaining cases can be obtained by reversing the direction of the flow accordingly to the corresponding spectra.}
	    \end{figure}
    
	\subsubsection*{Fold} 

	In this case the potential function is $V(x,a,b,c)=\frac{1}{3}x^3+ax$. The constraint manifold $\sv=\lbrace (x,a,b,c)\in\R^4 | x^2+a=0\rbrace$ is $3$-dimensional. The attracting part of $\sv$ is given by
	\eqn{
	\svm=\left\{ (x,a,b,c)\in\sv\,|\,x\geq0 \right\}
	}
	The projection $\tilde\pi=\pi|\sv$ is given by

	\eqn{
	\tilde\pi=\pi(x,-x^2,b,c)=(-x^2,b,c).
	}

	Note that the determinant of $\tilde\pi$ is non-positive for points in $\svm$. From this point we know that the trajectories of $\oX$ and of $X$ have opposite direction. Due to the presence of $3$ parameters, the fold set is the plane \[B=\lbrace (x,a,b,c)\in\R^4|(x,a)=(0,0) \rbrace.\]

	It is important to note that all phase portraits of the the Fold case have projections matching figure 3 of \cite{Takens1}.
	\begin{itemize}
	    \item Flow-box-1. By recalling the normal form in theorem \ref{teo:main} it is easy to see that the integral curves are as depicted in figure \ref{fig:fold-reg}.
	    \begin{figure}[!htbp]
	        \centering
	        \includegraphics[scale=0.8]{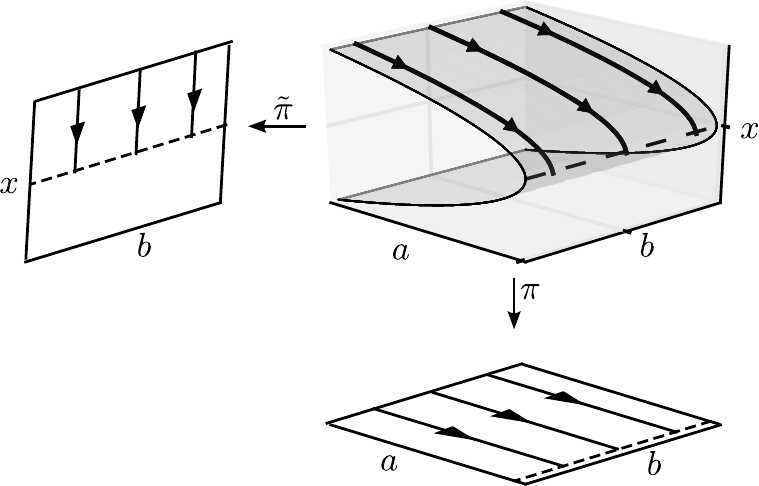}
	        \caption{Phase portrait and projections of the flow-box-1 case with the variable \ensuremath{c} suppressed.  The shown folded surface is a tomography of the \ensuremath{3}-dimensional constraint manifold \ensuremath{\sv}. The dotted line corresponds to the \ensuremath{2}-dimensional bifurcation set. Observe that since we are suppressing the variable \ensuremath{c}, this phase portrait is also shown in figure 3 of \protect\cite{Takens1}}
	        \label{fig:fold-reg}
	    \end{figure}
    
	    \item Flow-box-2.
	        The phase portrait in this case is as in figure \ref{fig:fold-reg}, just the direction of the trajectories is reversed.
    
	    \item Source, Sink and Saddle.\\
	    In all the following cases, a $1$-dimensional center manifold $W^{^C}$ appears within the fold surface. The choice of $\rho=\pm 1$ changes the direction of $W^{^C}$. In all the following pictures we set $\rho=1$. The direction of the integral curves of $\oX$ and of $(V,X)$ are in opposite direction since $\det(D\tilde \pi)$ is negative in $\svm$ \cite{Takens1}.
	    \begin{figure}[!htbp]
	        \centering
	        \subfloat[Source phase portraits]{\includegraphics[scale=0.85]{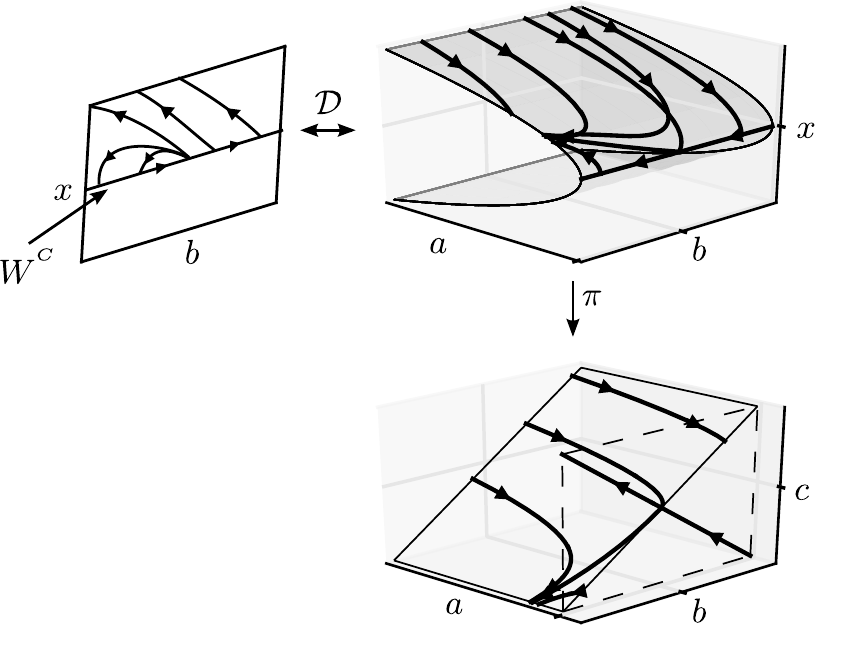}}\hfill
	        \subfloat[Sink phase portraits]{\includegraphics[scale=0.85]{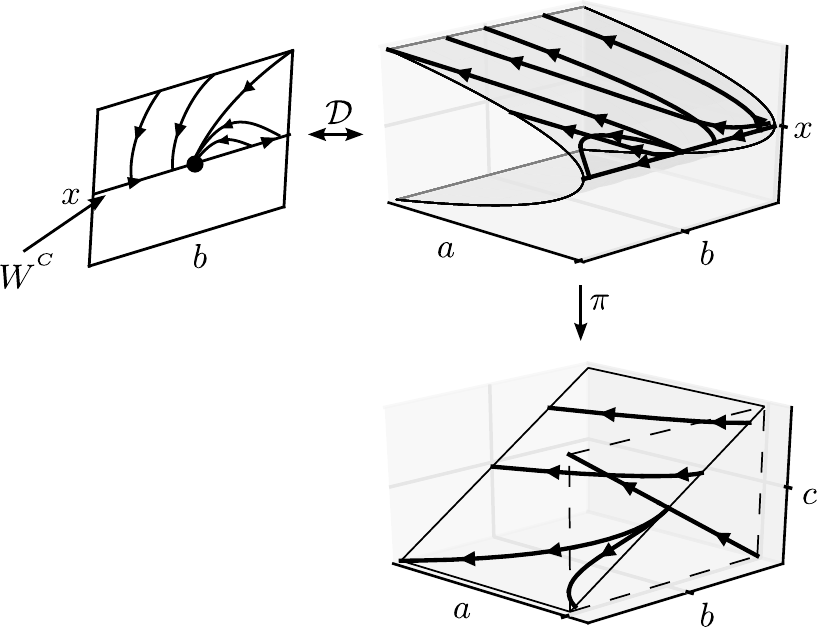}}\\
	        \subfloat[Saddle phase portraits]{\includegraphics[scale=0.85]{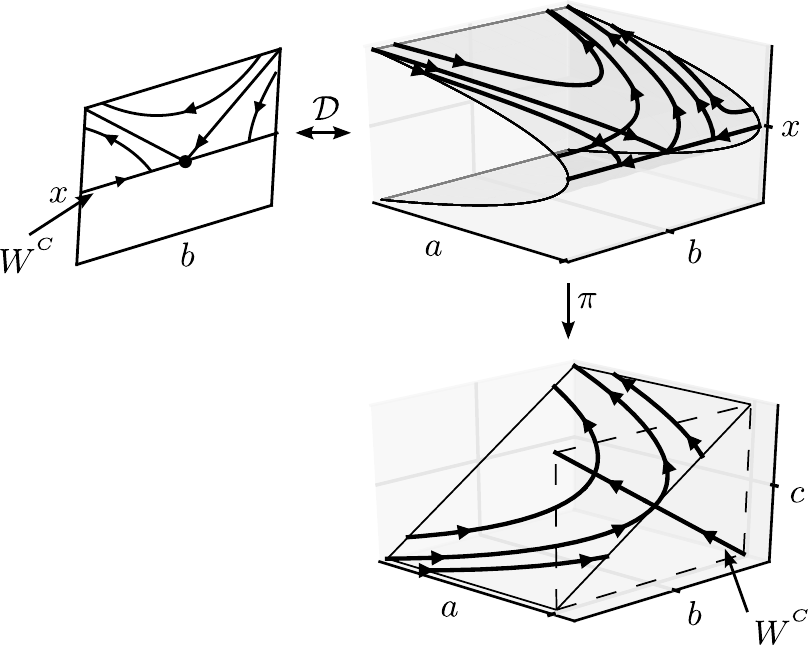}}
	        \caption{Projections of the solutions curves of the source, sink and saddle cases. The folded surface is a tomography (fixed value of \ensuremath{c}) of the \ensuremath{3} dimensional manifold \ensuremath{\sv}. The hyperplane \ensuremath{\lbrace {x,a,b,c}|b=c \rbrace} is invariant. In such space, the dynamics are reduced to the \ensuremath{2}-parameter fold listed in \cite{Takens1} and in theorem \protect\ref{teo:Takens}. Observe that there exists a \ensuremath{1}-dimensional manifold which is locally tangent to the fold surface.}
	        \label{fig:fold-others}
	    \end{figure}
	\end{itemize}
	\newpage
	\subsubsection*{Cusp}

	\begin{itemize}
	    \item The flow-box and the (dual) flow-box cases. Since in this case the generic vector field $\oX$ is a flow box, the phase portraits that we obtain are just the same as in Takens's list \cite{Takens1}. Just one more artificial variable, the $c$-coordinate, is considered.
    
	    \begin{figure}[!htbp]
	        \centering
	        \subfloat[Flow-box phase portraits]{\includegraphics[scale=0.95]{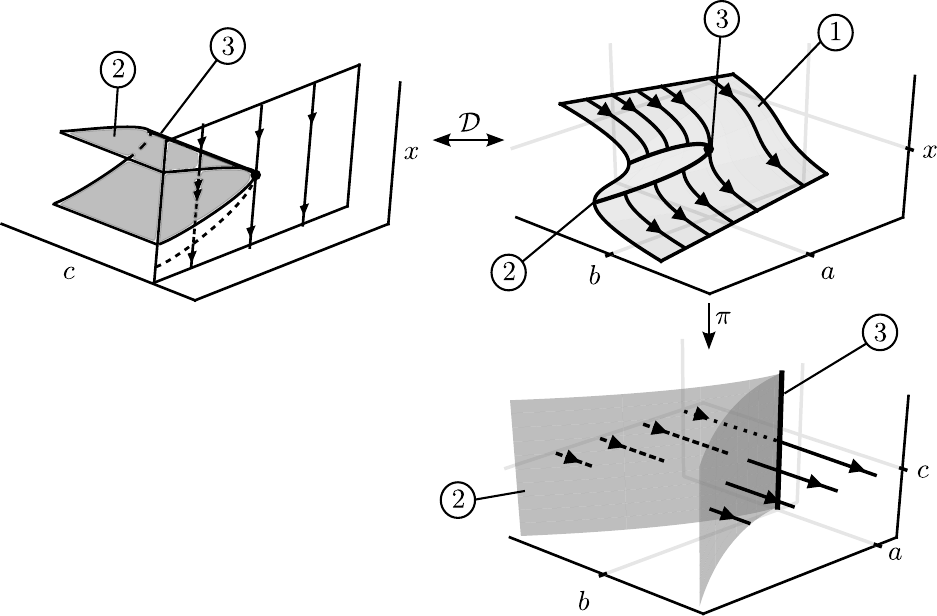}}\\
	        \subfloat[(Dual) Flow-box phase portraits]{\includegraphics[scale=0.95]{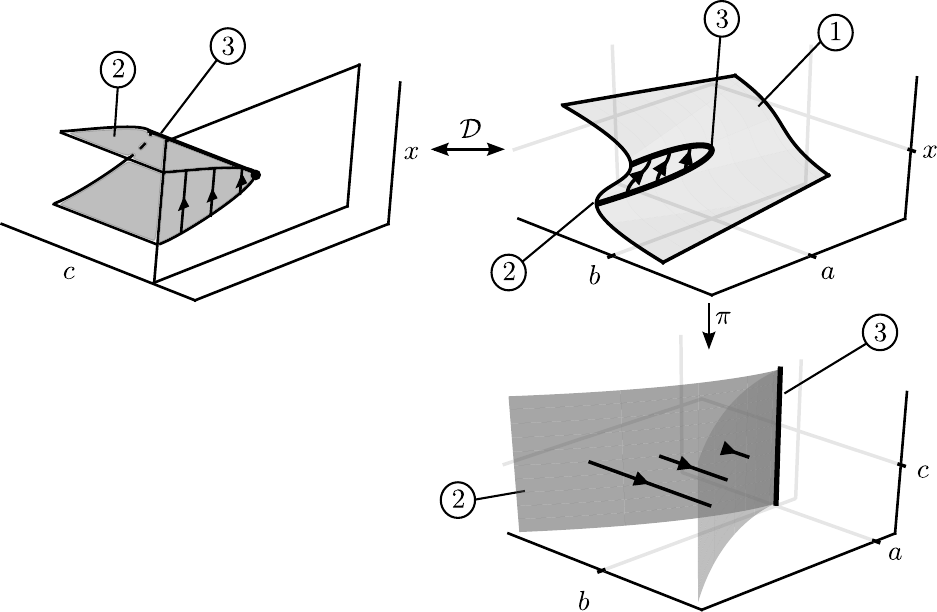}}
	        \caption{Phase portraits of the cusp (top) and the dual cusp (bottom) cases. \protect\circled{$1$} A tomography (the variable \ensuremath{c} is fixed and suppressed) of the \ensuremath{3}-dimensional manifold \ensuremath{\sv}.  \protect\circled{$2$} The \ensuremath{2}-dimensional fold manifold.  \protect\circled{$3$} The \ensuremath{1}-dimensional cusp manifold. Compare with \protect\cite{Takens1} figure 3 and note the resemblance with these projections. }
	        \label{fig:main-cusp}
	    \end{figure}
    
	\end{itemize} 
	\newpage
	\subsubsection*{Swallowtail}
	In this section we present the phase portrait of a generic CDE in a neighborhood of a swallowtail singularity. This is, we consider the potential function
	\eqn{
	V(x,a,b,c)=\dfrac{1}{5}x^5+\dfrac{1}{3}ax^3+\dfrac{1}{2}bx^2+cx.
	}
	Locally, the vector field is a flow-box and is depicted in figure \ref{fig:cde-pics-st}. It is straightforward to see that if one is to consider a potential function $-V$, the topology of the solutions does not change. Observe the jumping feature in the case $a<0$, see section \ref{sec:jumps} for more details on such phenomenon.
	\begin{figure}[!htbp]
	    \centering
	    \includegraphics[scale=0.95]{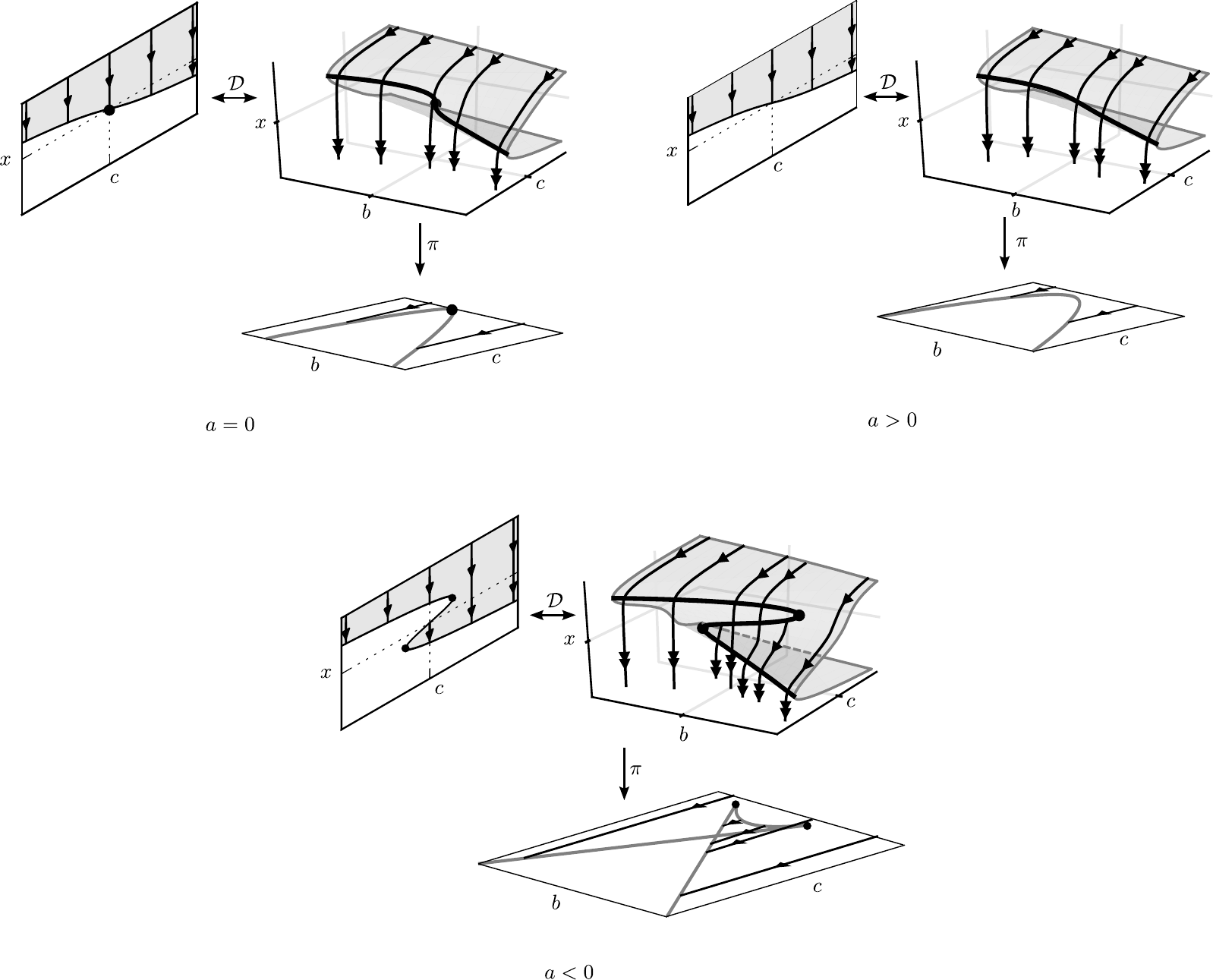}
	    \caption{Tomographies for different values of the parameter \ensuremath{a} of the phase portraits of the swallowtail case. The catastrophe is stratified in the sets shown in figure \ref{fig:st-stratification}. Note the particular behavior of the solutions when \ensuremath{a<0}. In such case, there exists a region near the origin where jumps may occur. Observe that the shown solutions are in accordance with our description is section \protect\ref{geo:st}, that is \ensuremath{X} is transverse to the projection of the singular set.}
	\end{figure}
	\label{fig:cde-pics-st}
	\newpage
	\subsubsection*{Hyperbolic Umbilic} 

	The total space is $\R^5$. The constraint manifold and the bifurcation set are detailed in figure \ref{fig:hu-stratification}. From the exposition of section \ref{sec:proof} we know that the origin of the desingularized vector field is an equilibrium point. We show in figures \ref{fig:hu-cs} and \ref{fig:hu-cc} the phase portraits of the center-saddle and center-center cases respectively. We take advantage on the fact that $\lbrace a=0\rbrace$ is an invariant set. This means that the integral curves are arranged by those in the subspace $(x,y,0,b,c)$. Note that both phase portraits satisfy the geometric description given in section \ref{geo:hu}. That is, the integral curves are transversal to the singular sets. We have decided to show only the solution curves within $\svm$ as those are the ones we are interested in.

	\begin{figure}[!htbp]
	    \centering
	    \includegraphics[scale=1]{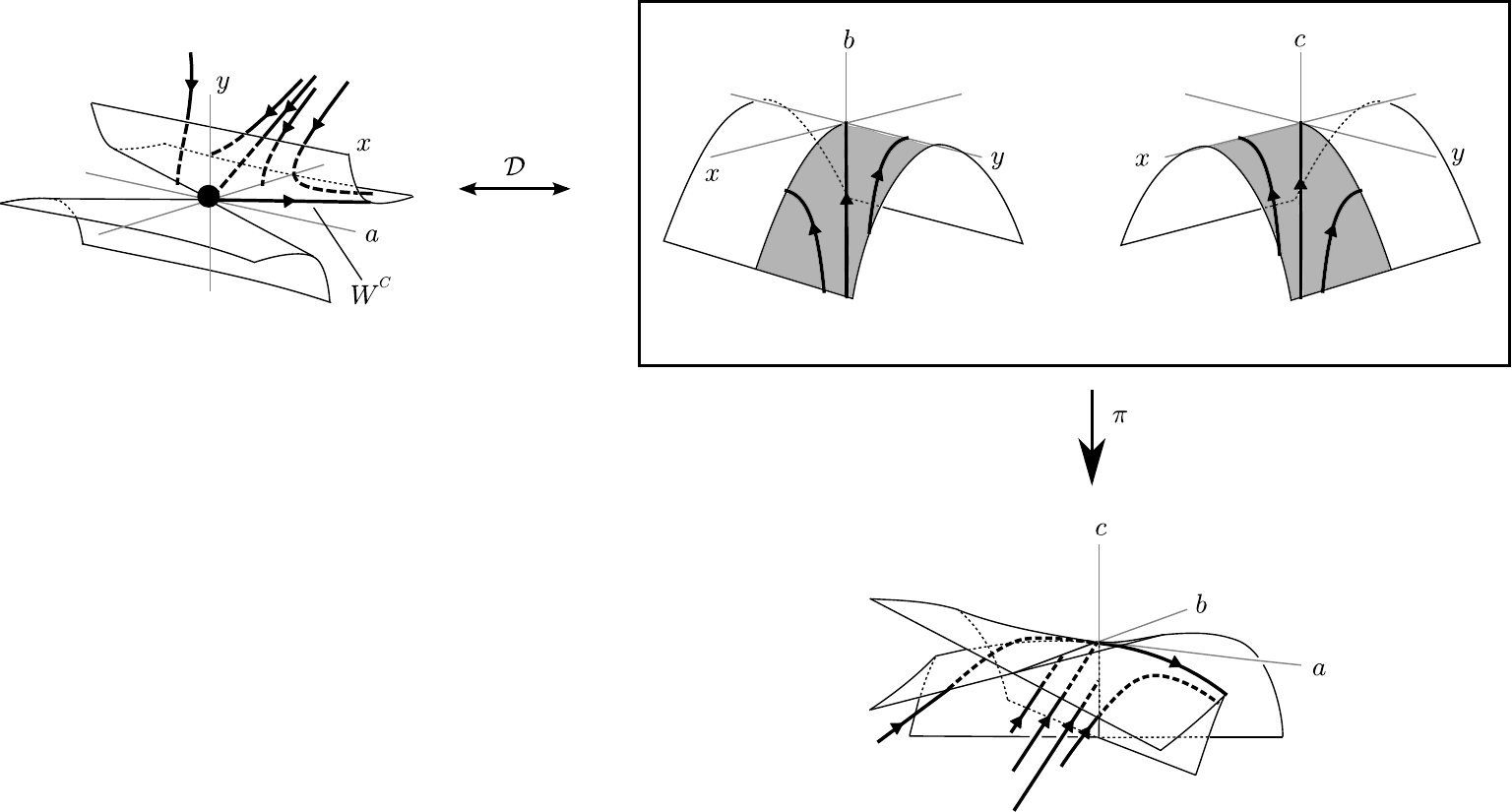}
	    \caption{ Phase portraits of the center-saddle case of the hyperbolic umbilic. Top left: the desingularized vector field. The origin is a semihyperbolic equilibrium point. Two directions correspond to a saddle, and one to a center manifold. Locally, such manifold is tangent to the singularity cone depicted. The center manifold changes direction depending on the \ensuremath{\pm} sign of the normal form. The trajectories shown are within the projection of \ensuremath{S_{V,min}}. Top right: Trajectories of the CDE \ensuremath{(V,X)} restricted to \ensuremath{\svm}. The latter set is shown as a shaded region. Bottom: the projection of the solution curves into the parameter space. Note that the phase portrait shown satisfy the conjecture given in section \protect\ref{geo:hu}. }
	    \label{fig:hu-cs}
	\end{figure}

	\begin{figure}[!htbp]
	    \centering
	    \includegraphics[scale=1]{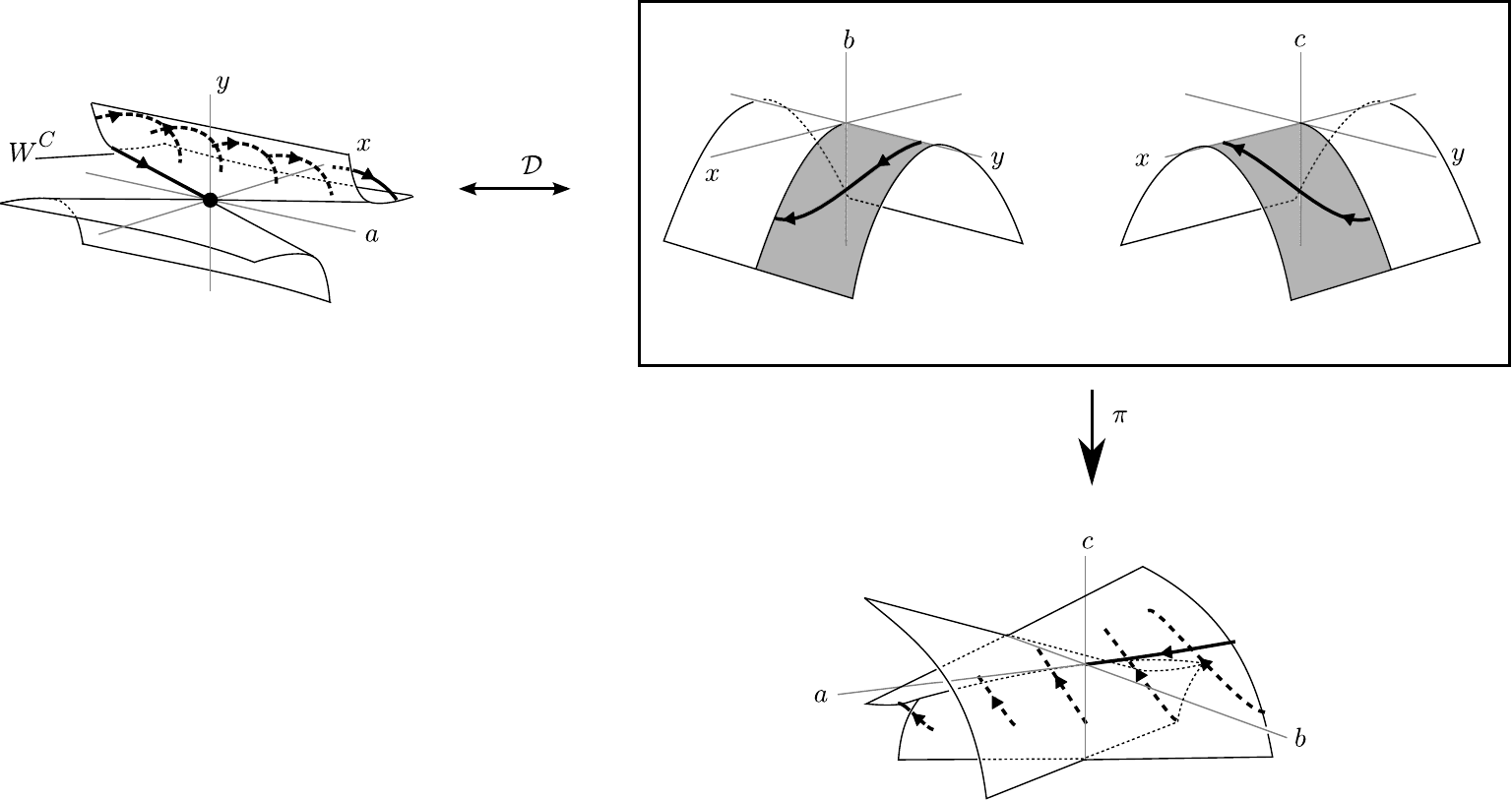}
	    \caption{ Phase portraits of the center case of the hyperbolic umbilic singularity. Top left: the desingularized vector field. Such vector field has an equilibrium at the origin and a 3-dimensional center manifold. The direction of the 1-dimensional center manifold depicted changes according to the \ensuremath{\pm} sign of the normal form. Top right: Solutions curves in the invariant space \ensuremath{\svm|a=0}. The latter set is shown as a shaded region. Bottom: the projection of the solution curves into the parameter space.  Note that the phase portrait shown satisfy the conjecture given in section \protect\ref{geo:hu}.}
	    \label{fig:hu-cc}
	\end{figure}
	\newpage
	\subsubsection*{Elliptic Umbilic} The constraint manifold and the bifurcation set are described in figure \protect\ref{fig:hu-stratification}. We show in figure \protect\ref{fig:eu-cs} the phase portrait of the center-saddle. It is easy to check that \ensuremath{\svm|a=0} is just a point, so unlike in the hyperbolic umbilic case, there are no solutions curves of the corresponding CDE at \ensuremath{\lbrace a=0\rbrace}. Therefore, we show projections into \ensuremath{\svm|a>0} with the value of \ensuremath{a} fixed, of some integral curves. 

	\begin{figure}[!htbp]
	    \centering
	    \includegraphics[scale=1]{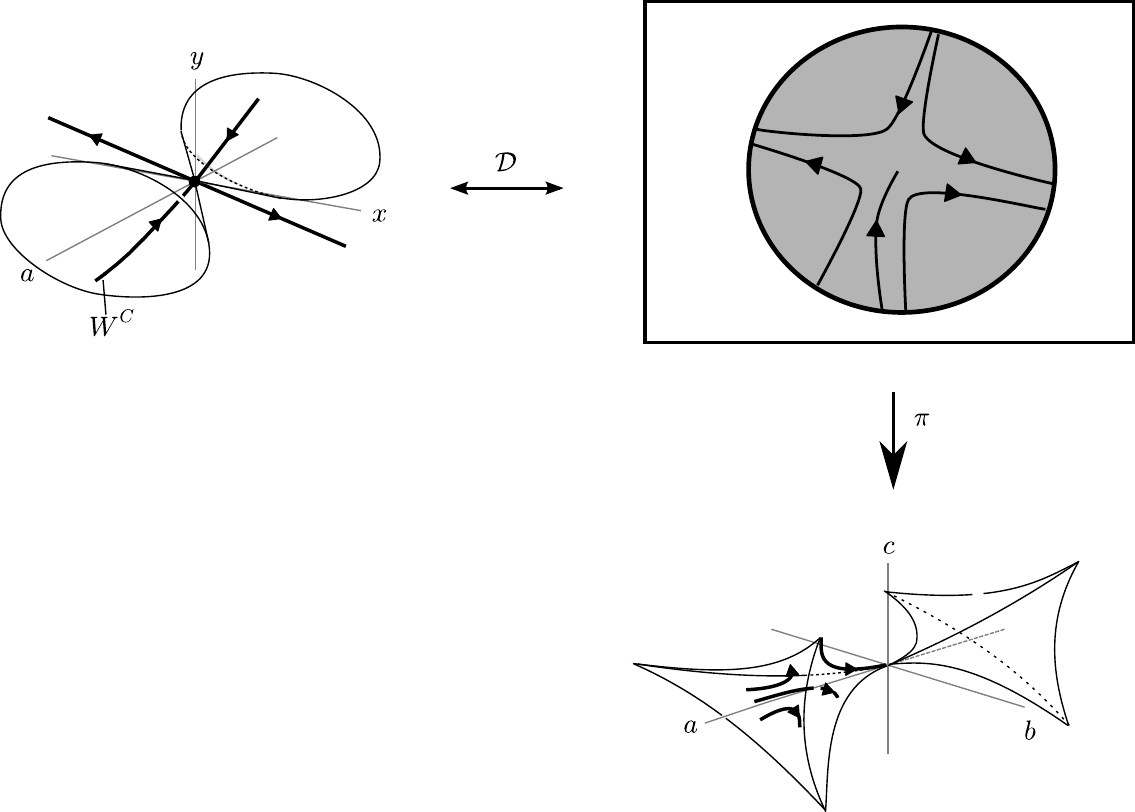}
	    \caption{ Phase portraits of the center-saddle case of the elliptic umbilic. Top left: the desingularized vector field. The origin is a semi-hyperbolic equilibrium point with two hyperbolic and one center directions. The center manifold is locally tangent to the singularity cone depicted. The hyperbolic directions shown (corresponding to a saddle) together with the center manifold arrange all the integral curves sufficiently close to the origin. Top right: Projection of some solutions curves into a tomography (\ensuremath{a} fixed) of \ensuremath{\svm}. Observe that \ensuremath{\svm} is the inside region of a cone (refer to figure \protect\ref{fig:eu-stratification} and section \protect\ref{geo:eu}). Bottom: the projection of the solution curves into the parameter space.}
	    \label{fig:eu-cs}
	\end{figure}

    \newpage
    \subsection{Jumps in generic CDEs with three parameters}\label{sec:jumps}
	    Constrained differential equations and slow-fast systems are closely related. CDEs may represent an approximation of some generic dynamical systems with two or more different time scales. One interesting behavior of the latter type of systems is formed by jumps. Roughly speaking a jump is a rapid transition from one stable part of $\sv$ to another. One common example of such behavior a relaxation oscillations. See also the examples in section \ref{sec:motivation}, where the characteristic property of jumps is described. \\
    
	    In this section we discuss the possibility of encountering such jumping behavior in generic CDEs with a swallowtail, hyperbolic, or elliptic umbilic singularity.
    
	    \begin{definition}[\sc Finite Jump] Let $\gamma$ be a solution curve of a CDE $(V,X)$. Let $q\in B$. We say that $\gamma$ has a finite jump at $q$ if
	        \begin{enumerate}
	        \item There exists a point $p\in \svm$ such that $\pi(p)=\pi(q)$.
	        \item There exists a curve from $p$ to $q$ along which $V$ is monotonically decreasing.
	        \end{enumerate}
	    \end{definition}
    
	    In the case of the fold singularity, there are no finite jumps. In the case of the cusp singularity, a solution curve $\gamma$ has the jump \cite{Takens1}
    
	    \eq{
	    (x,a,b)\to(-2x,a,b).
	    }
    
	    To study if there exist finite jumps in the generic CDEs with three parameters, we have the following proposition.
    
	    \begin{proposition}[Jumps in CDEs with 3 parameters] Let $(V,X)$ be a generic CDE with potential function $V$ one of the codimension 3 catastrophes. Let $\gamma$ be a solution curve of $(V,X)$. Then
	        \begin{enumerate}
	            \item If $V$ is the swallowtail catastrophe, then there are finite jumps as follows. Let $(x,a,b,c)$ be coordinates of $\gamma\cap B$, then the finite jump is given by
            
	             \eq{
	             	(x,a,b,c)\mapsto(-x-\sqrt{-2x^2-a},a,b,c),
	             }
	             where it is readily seen that 
	             \eq{
	    x\in \left( -\sqrt{ -\dfrac{a}{2} }, \sqrt{-\dfrac{a}{2} } \right), \; a<0.
	}
	            \item If $V$ is the hyperbolic or the elliptic umbilic catastrophe, then there are no finite jumps.
	        \end{enumerate}
        
	    \end{proposition}

	    \begin{proof} We detail the proof of the hyperbolic umbilic case. The other cases follow the same methodology.\\
        
	        Recall that for the hyperbolic umbilic
        
	        \eq{
	        	\sv = \left\{ (x,y,a,b,c)\in\R^5 \; | \; b=-3x^2-ay, \; c=-3y^2+ax \right\},
	        }
	        \eq{\label{eq:hu-svm-proof}
	        	\svm = \left\{ (x,y,a,b,c)\in\sv \; | \;  36xy-a^2 \geq 0, \; x+y>0 \right\},
	        }
	        and
	        \eq{
	        	B = \left\{ (x,y,a,b,c)\in\sv \; | \; 36xy-a^2 =0\right\}.
	        }
        
	        Let $p=(x_1,y_1,a_1,b_1,c_1)\in S_V$ and $q=(x_2,y_2,a_2,b_2,c_2)\in B$. So we have that the projections $\pi(p)$ and $\pi(q)$ read
	        \begin{equation}
	        \begin{split}
	        \pi(p)&=(a_1, -3x_1^2-a_1y_1, -3y_1^2-a_1x_1)\\
	        \pi(q)&=(a_2, -3x_2^2-a_2y_2, -3y_2^2-a_2x_2), \quad a_2^2=36x_2y_2.
	        \end{split}
	        \end{equation}

	        The point $q=\gamma\cap B$ is known. The point $p$ is unknown, it corresponds to a possible arriving point when a finite jump occurs. If such a point $p$ exists, then it is a nontrivial solution of $\pi(p)=\pi(q)$. The easiest case is when $a_2=0$. We have
	        \begin{equation}
	        \begin{split}
	        \pi(p)&=(0, -3x_1^2, -3y_1^2)\\
	        \pi(q)&=(0, -3x_2^2, -3y_2^2), \quad 0=x_2y_2.
	        \end{split}
	        \end{equation}

	        Here we have two cases: 1) $0=x_2y_2 \implies x_2=0, \text{ and } y_2\neq 0$, or
	        2) $0=x_2y_2 \implies x_2\neq 0, \text{ and } y_2= 0$.
	        \begin{enumerate}
	        \def\labelenumi{\arabic{enumi}.}
	        \item
	          $a_2=0, \, x_2=0, \, y_2\neq 0$. We have
	          \begin{equation}
	              \begin{split}
	              -3x_1^2&=0\\
	              -3y_1^2&=-3y_2
	              \end{split}
	              \end{equation}

	          The non trivial solution is $(x_1,y_1)=(0,-y_2)$. So, there is a possible finite jump of the form
	          \eq{
	          q_1=(0,y_2,0,b_2,c_2) \mapsto p_1=(0,-y_2,0,b_2,c_2)
	          }
	        \item
	          $a_2=0, \, x_2\neq 0, \, y_2= 0$. Similarly we have the possible jump
	          \eq{
	          q_2=(x_2,0,0,b_2,c_2)\mapsto p_2=(-x_2,0,0,b_2,c_2).
	          }
	        \end{enumerate}

	        Now we check if any of such arriving points are in
	        $S_{V,min}$. The conditions for a
	        point $p=(x,y,a,b,c)$ to be in $S_{V,min}$ are
	        \begin{equation}
	        \begin{split}
	        -3x^2-ay-b&=0\\
	        -3y^2-ax-c&=0\\
	        36xy-a^2 &\geq 0\\
	        x+y &\geq 0.
	        \end{split}
	        \end{equation}

	        It is readily seen then that for $a=0$, $p_1$ and $p_2$ are not points in $\svm$ as the last inequality is not satisfied.\\

	        Now, we study the case $a_2\neq 0$. The problem $\pi(p)=\pi(q)$ can be rewritten as the nonlinear
	        simultaneous equation
	        \begin{equation}
	        \begin{split}
	        -3x_1^2-a_2y_1+3x_2^2+a_2y_2&=0\\
	        -3y_1^2-a_2x_1+3y_2^2+a_2x_2&=0.
	        \end{split}
	        \end{equation}

	        Since $a_2\neq 0$ we can write from the first equation
	        \begin{equation}
	        \begin{split}
	        y_1=\frac{-3x_1^2+3x_2^2+a_2y_2}{a_2},
	        \end{split}
	        \end{equation}

	        substituting in the second equation we get
	        \begin{equation}
	        \begin{split}
	        27x_1^4-(54x_2^2+18a_2y_1)x_2^2+a_2^3x_1+18x_2^2a_2y_2-a_2^3x_2+27x_2^4=0.
	        \end{split}
	        \end{equation}

	        It is not difficult to see that $x_1=x_2$ is a double root, so we have the factorization
	        \begin{equation}
	        \begin{split}
	        (x_1-x_2)^2(3x_1^2+6x_2x_1+3x_2^2-2a_2y_2)=0.
	        \end{split}
	        \end{equation}

	        The roots of $3x_1^2+6x_2x_1+3x_2^2-2a_2y_2=0$ are
	        \begin{equation}
	        \begin{split}
	        X_\pm&=-x_2\pm\frac{2}{\sqrt{6}}\sqrt{a_2y_2}.
	        \end{split}
	        \end{equation}

	        The corresponding $y_1$ solutions are
	        \begin{equation}
	        \begin{split}
	        Y_\pm&=-y_2\pm 2\sqrt{6} x_2 \sqrt{\frac{y_2}{a_2}}.
	        \end{split}
	        \end{equation}

	        This is, for a trajectory $\gamma$ such that
	        $\gamma\vert B=(x_2,y_2,a_2,b_2,c_2)$, there are possible jumps to
	        \begin{equation}
	        \begin{split}
	        p_1&=\left( -x_2+\frac{2}{\sqrt{6}}\sqrt{a_2y_2}, -y_2+\frac{2\sqrt{6}\sqrt{y_2}x_2}{\sqrt{a_2}}, a_2,b_2,c_2 \right)\\
	        p_2&=\left( -x_2-\frac{2}{\sqrt{6}}\sqrt{a_2y_2}, -y_2-\frac{2\sqrt{6}\sqrt{y_2}x_2}{\sqrt{a_2}}, a_2,b_2,c_2 \right).
	        \end{split}
	        \end{equation}

	        Just as in the previous case, we shall check if the points
	        $(X_+,Y_+,a_2,b_2,c_2)$, $(X_-,Y_-,a_2,b_2,c_2)$ are contained in
	        $S_{V,min}$. This is, we have to check if the following inequalities are
	        satisfied.
	        \begin{equation}\label{eq:ineq1}
	        \begin{split}
	        X_++Y_+&\geq 0\\
	        36X_+Y_+-a_2^2&\geq 0
	        \end{split}
	        \end{equation}
	        and
	        \begin{equation}\label{eq:ineq2}
	        \begin{split}
	        X_-+Y_-&\geq 0\\
	        36X_-Y_--a_2^2&\geq 0
	        \end{split}
	        \end{equation}

	        In both cases we have the further properties $36x_2y_2-a_2^2=0$ and
	        $x_2+y_2\geq 0$ (recall that $(x_2,y_2,a_2,b_2,c_2)\in B$). By substituting the value
	        $y_2=\frac{a^2}{36x_2}$ in $X_{\pm}$ and $Y_\pm$ we have
	        \begin{equation}
	        \begin{split}
	        X_\pm&=-x_2\pm\frac{a_2^{3/2}}{3\sqrt{6}x_2^{1/2}}\\
	        Y_\pm&=-\frac{a_2^2}{36x_2}\pm\frac{2}{\sqrt{6}}x_2^{1/2}a_2^{1/2}
	        \end{split}
	        \end{equation}
        
	        Now, \eqref{eq:ineq1} and \eqref{eq:ineq2} read
	        \begin{equation}\label{eq:ineq1b}
	        \begin{split}
	        -x_2+\frac{a_2^{3/2}}{3\sqrt{6}x_2^{1/2}}-\frac{a_2^2}{36x_2}+\frac{2}{\sqrt{6}}x_2^{1/2}a_2^{1/2}&\geq 0\\
	        \left( -x_2+\frac{a_2^{3/2}}{3\sqrt{6}x_2^{1/2}}  \right)\left( -\frac{a_2^2}{36x_2}+\frac{2}{\sqrt{6}}x_2^{1/2}a_2^{1/2} \right)-a_2^2 &\geq 0
	        \end{split}
	        \end{equation}

	        and
	        \begin{equation}\label{eq:ineq2b}
	        \begin{split}
	        -x_2-\frac{a_2^{3/2}}{3\sqrt{6}x_2^{1/2}}-\frac{a_2^2}{36x_2}-\frac{2}{\sqrt{6}}x_2^{1/2}a_2^{1/2}&\geq 0\\
	        \left( -x_2-\frac{a_2^{3/2}}{3\sqrt{6}x_2^{1/2}}  \right)\left( -\frac{a_2^2}{36x_2}-\frac{2}{\sqrt{6}}x_2^{1/2}a_2^{1/2} \right)-a_2^2 &\geq 0
	        \end{split}
	        \end{equation}

	        respectively. It is readily seen that \eqref{eq:ineq2b} is not satisfied. Now we focus on \eqref{eq:ineq1b}. First we check the conditions for $X_+\geq 0$ and $Y_+\geq 0$. We have

	        \[ X_+\geq 0 \implies x_2^{3/2} \leq \frac{a^{3/2}}{3\sqrt{6}}, \]

	        \[ Y_+\geq 0 \implies  \frac{1}{12\sqrt{6}}a^{3/2}\leq x_2^{3/2}.\]

	        This is $\frac{1}{12\sqrt{6}}a_2^{3/2}\leq x_2^{3/2}\leq\frac{1}{3\sqrt{6}}a_2^{3/2}$.
	        Of course this would imply that $X_++Y_+\geq 0$. Now we have to check if
	        for such interval $36X_+Y_+-a^2\geq 0$.

	        \[
	        36X_+Y_+-a^2=12\sqrt{6}x_2^{3/2}a_2^{1/2}-\frac{a_2^{7/2}}{3\sqrt{6}x_2^{3/2}}+4a^2,
	        \]

	        so we check if
	        \eq{
	        -12\sqrt{6}x_2^{3/2}a_2^{1/2}-\frac{a_2^{7/2}}{3\sqrt{6}x_2^{3/2}}+4a^2\geq 0
	        }
	        in the interval
	        \eq{\label{eq:int}
	        \frac{1}{12\sqrt{6}}a_2^{3/2}\leq x_2^{3/2}\leq\frac{1}{3\sqrt{6}}a_2^{3/2}.
	        }
	        We have that
	        \[ 12\sqrt{6}x_2^{3/2}a_2^{1/2}+\frac{a_2^{7/2}}{3\sqrt{6}x_2^{3/2}}=\frac{216x_2^3a_2^{1/2}+a_2^{7/2}}{3\sqrt{6}x_2^{3/2}},\]

	        but note that from \eqref{eq:int} we obtain
	        \[
	        \frac{1}{4}a_2^{3/2}\leq 3\sqrt{6}x_2^{3/2},
	        \]

	        so we have
	        \[
	        \frac{216x_2^3a_2^{1/2}+a_2^{7/2}}{3\sqrt{6}x_2^{3/2}}\geq 864x_2^{3/2}a_2^{-3/2}x_2^{3/2}a_2^{1/2}+4a_2^2\geq\frac{72}{\sqrt{6}}a_2^{1/2}x_2^{3/2}+4a_2^2\geq{5a_2^2}.
	        \]

	        This means that the inequality \[
	        -12\sqrt{6}x_2^{3/2}a_2^{1/2}-\frac{a_2^{7/2}}{3\sqrt{6}x_2^{3/2}}+4a^2\geq 0
	        \]

	        can not be satisfied, which implies that $\pi(p)=\pi(q)$ does not have nontrivial solutions in $\svm$. Therefore, it is not possible to have finite jumps.

	    \end{proof}
\appendix
\renewcommand*{\appendixname}{}

\section{Thom-Boardman symbol}\label{app:TB}
    Let $N^n,\, M^m$ be smooth manifolds, and consider that $(x_1,\ldots,x_n)$ and $(y_1,\ldots,y_m)$ are some local coordinates in $N$ and $M$ respectively. Let a smooth map $f:N^n\to M^m$ be given by $y_i=f_i(x)$. Let $i_1$ be a nonnegative integer. The set $\Sigma^{i_1}(f)$ consists of all points at which the kernel of $Df$ has dimension $i_1$. Given a finite sequence $I=(i_1,i_2,\ldots,i_k)$ of non-increasing nonnegative numbers, $\Sigma^I(f)$ is defined inductively as follows.

\begin{definition}[\sc Thom-symbol] Assume that $\Sigma^I(f)=\Sigma^{i_1,i_2,\ldots,i_k}(f)\subset N$ is a smooth manifold. Then 
    \eqn{
    \Sigma^{i_1,i_2,\ldots,i_k,i_{k+1}}=\Sigma^{i_{k+1}}\left( f|\Sigma^ I(f)\right)
    }
    
    is the set of all points at which the kernel of $D(f|\Sigma^ I(f))$ has dimension $i_{k+1}$.
    
\end{definition}
Naturally, we have the inclusions
\eqn{
N\supset\Sigma^{i_1}(f)\supset\Sigma^{i_1,i_2}(f)\supset\cdots
}

Denote by $\mathscr{E}_n$ the ring of germs of $\C^\infty$ functions on $\R^n$ at $0$. Let $\mathscr{I}$ be an ideal of $\mathscr{E}_n$.

\begin{definition}[\sc Jacobian extension] The \emph{Jacobian extension} $\Delta_k(\mathscr{I})$ of $\mathscr{I}$ is the ideal generated by $\mathscr{I}$ and all the Jacobians $\det\left(  \parcs{\phi_i}{x_j}\right)$ of order $k$, and where $\phi_i$ are functions in $\mathscr{I}$.    
\end{definition}

\begin{remark} $ $
    \begin{itemize}
        \item The ideal $\Delta_k(\mathscr{I})$ is independent on the choice of coordinates.
        \item $\Delta_{k+1}(\mathscr{I})\subseteq\Delta_k(\mathscr{I})$.
    \end{itemize}
\end{remark}

\begin{definition}[\sc Critical Jacobian extension] A Jacobian extension $\Delta_k(\mathscr{I})$ is said to be \emph{critical} if $\Delta_{k}\neq\mathscr{E}_n$ but $\mathscr{E}_n=\Delta_{k-1}(\mathscr{I})$. This is, the order $k$ of the Jacobians is the smallest for which the extension does not coincide with $\mathscr{E}_n$.
\end{definition}

Now, we change lower indices to upper indices as follows. 
\begin{definition}\label{def:up}
    $\Delta^k=\Delta_{n-k+1}$.
\end{definition}

By using the upper indices as in definition \ref{def:up}, we have that 

\eqn{
i_1=\cor(\mathscr{I}), \; i_2=\cor(\Delta^{i_1}\mathscr{I}), \; \ldots, \; i_k=\cor(\Delta^{i_{k-1}}\cdots \Delta^{i_k}\mathscr{I}).
}

\begin{definition}[\sc Thom-Boardman symbol] Let $I=(i_1,i_2,\ldots,i_k)$ be a non-increasing sequence of non-negative integer numbers. The ideal $\mathscr{I}$ is said to have \emph{Thom-Boardman symbol $I$} if its successive critical extensions are 
    
    \[\Delta^{i_1}\mathscr{I},\, \Delta^{i_2}\Delta^{i_1}\mathscr{I}, \ldots, \Delta^{i_k}\Delta^{i_{k-1}}\cdots\Delta^{i_2}\Delta^{i_1}\mathscr{I}.\]
    
\end{definition}

\begin{definition} [\sc Symbol of a Singularity]
    Let the map $f:N^n\to M^m$ be such that $f(0)=0$. We say that \emph{ $f$ has a singularity of Thom-Boardman symbol $\Sigma^I$ at $0$} if the ideal generated by the $m$ coordinate functions $f_i$ has Thom-Boardman symbol $I$.
\end{definition}

\begin{definition}[\sc Nice map] A map $f$ is said to be \emph{nice} if its $k$-jet extension is transverse to the manifolds $\Sigma^I$.
    
\end{definition}

The importance of a nice map is contained in the following result.

\begin{theorem}[\sc On nice maps]\cite{Boardman} $ $
    
    \begin{enumerate}
        \item If $f:N^n\to M^m$ is a nice map, then $\Sigma^I(f)=(j^kf)^{-1}(\Sigma^I)$. This is $\Sigma^I(f)$ is a submanifold of $N$ and $x\in\Sigma^I(f)$ if and only if $j^kf(x)\in\Sigma^I$.
        \item Any smooth map $f:N^n\to M^m$ can be arbitrarily well approximated by a nice map.  
    \end{enumerate} 
\end{theorem}

\section{Desingularization}\label{app:des}
     Note that we can write each elementary catastrophe in the form
    \eq{
    	V(x,a) = V(x,0) + \sum_{i=1}^m a_i \parcs{V(x,a)}{a_i},
    }
     where $x\in\R^n,\; a\in\R^m$, and with $n \leq m$. The constraint manifold (see definition \ref{def:CDE}) is given by $\parc{x}V(x,a)=0$, which means
    \eq{
    	\parcs{V(x,0)}{x_j}+ \sum_{i=1}^m a_i \parcs{^2V(x,a)}{x_j\partial a_i}=0, \qquad \forall j\in [1,n].
    }

    Next, note that we can always solve the previous equation for $n$ of the $a_j$'s, obtaining
    \eq{
    	a_{j}=-\parcs{V(x,0)}{x_j}-\sum_{i=j+1}^m a_i \parcs{^2V(x,a)}{x_j\partial a_i}.
    }	

    This expresses that $a_j$ is the coefficient of the linear term $x_j$ in the potential function $V(x,a)$. Now, we can choose coordinates in $S_V$ as 
    \eq{
    \left(x_1,\ldots,x_n,-\parcs{V(x,0)}{x_1}-\sum_{i=j+1}^m a_i \parcs{^2V(x,a)}{x_1\partial a_i},\ldots,-\parcs{V(x,0)}{x_n}-\sum_{i=j+1}^m a_i \parcs{^2V(x,a)}{x_n\partial a_i},a_{n+1},\ldots,a_m \right).
    }

    Next, we define the projection $\tilde \pi=\pi|\sv$, this is
    \eq{
    	\tilde\pi =\left(-\parcs{V(x,0)}{x_1}-\sum_{i=j+1}^m a_i \parcs{^2V(x,a)}{x_1\partial a_i},\ldots,-\parcs{V(x,0)}{x_n}-\sum_{i=j+1}^m a_i \parcs{^2V(x,a)}{x_n\partial a_i},a_{n+1},\ldots,a_m \right).
    }

    In the original coordinates, $X$ has the general form
    \eq{
    	X=\sum_{i=1}^m f_i(x,a)\parc a_i.
    }

    $S_V$ is the phase space of a constrained differential equation. So, for a point in $S_V$ with coordinates $(x_1,\ldots,x_n,a_{n+1},\ldots,a_m)$, $\tilde X$ is given by
    \eq{
    	\tilde X = (d\tilde\pi)^{-1}X(x,\tilde\pi(x,a)).
    }

    It is clear that $\tilde X$ is defined only for points where the projection in non-singular. Next, recall that the map $A\mapsto \det(A)A^{-1}$ can be extended to a $\C^\infty$ map on the space of square matrices. This means that we can define a smooth vector field by
    \eq{
    	\oX = \det(d\tilde\pi)(d\tilde\pi)^{-1}X(x,\tilde\pi(x,a)).
    }

    Note that for all points where $\det(d\tilde\pi)\neq 0$, the solutions of $(V,X)$ are obtained from the integral curves of $\oX$. First by reparametrization due to the smooth projection $\tilde\pi$, and in cases where $\det(d\tilde\pi)<0$, by then reversing the direction of the solutions.

\section{Center Manifold Reduction}\label{app:cm}
Let a $\C^\infty$ vector field $Y(x)$  be given as $Y=\sum_{i=1}^nf_i\parc{x_i}$. Assume that the origin is an isolated equilibrium point, this is $Y(0)=0$. Assume also that the Jacobian of $Y$ has $c$ eigenvalues in the imaginary axis, and let $\ell$ be a positive integer. We have the following result.

  \begin{theorem}[{\sc{Center manifold}}]\label{teo:cm} There exists a $\C^\ell$, $c-$dimensional manifold $W^{c}$ containing the origin, and a neighborhood $U$ of $0\in\R^n$, such that for any point $x\in W^{c}\cap U$, $Y(x)$ is tangent to $W^{c}$ at $x$. Moreover, there exists an integer $r$, with $0\leq r\leq n-c$ such that $Y$ is \emph{topologically equivalent} to the vector field
  \eq{
    \oY=\sum_{i=1}^{c}\tilde f_i(y_1,\ldots,y_{c})\parc{y_i}+\sum_{i=c+1}^{c+r}y_i\parc{y_i}-\sum_{i=c+r+1}^ny_i\parc{y_i},
  }
  where $(y_1,\cdots,y_{c})$ are coordinates in the center manifold $W^{c}$, and all eigenvalues of $D_0\tilde f$ are on the imaginary axis.
  \end{theorem}

It is important to note that the center manifold $W^{c}$ in theorem \ref{teo:cm} is not unique. However, different choices of $W^c$ lead to topologically equivalent phase portraits \cite{Arnold2, Takens2}.

\section{Takens's Normal Form Theorem}\label{app:nf}
Assume $Y(x)$ is a vector field as above in item appendix \ref{app:cm}. The purpose of the following theorem is to write the vector field $Y$ in it's $k-$jet, and in some simple form. For this, define by $Y_1(x)$ the vector field which has the same $1-$jet at $x=0$ as $Y$, and whose coefficients are linear in $x$. Denote by $\mathcal H^k$ the space of vector fields whose coefficients are homogeneous polynomials of degree $k$.\\

  The linear map $[Y_1,-]_k:\mathcal H^k\to\mathcal H^k$ assigns to each $H\in\mathcal H^k$ the Lie product $[Y_1,H]$. Observe that for a fixed $Y_1$ there is a splitting $\mathcal H^k=B^k+G^k$,  where $B^k={\rm{Im}}\left([Y_1,-]_k\right)$, and $G^k$ is some complementary space.

  \begin{theorem}[{\sc {Normal form theorem \cite{Takens2}}}]\label{teo:nf} Let $Y,\; Y_1, \; B^k, \; G^k$ be as above. Then, for $\ell \leq k$, there exists a $\C^\infty-$diffeomorphism $\phi:\R^n\to\R^n$, which fixes the origin, such that $\phi_*(Y)=Y'$ is of the form
  \eq{
  Y'=Y_1+g_2+\cdots+g_\ell+R_\ell
  }    
  where $g_j\in G^j$, $j=2,\ldots,\ell$ and $R_\ell$ is a vector field with vanishing $\ell-jet$ at the origin, $\ell=k=\infty$ is not excluded.
  \end{theorem}
  \begin{remark} In case the $1-jet$ of $Y$ is identically $0$, one proceeds as follows. Let $s$ be the smallest integer such that the $s-jet$ of $Y$ does not vanish at $0$, denote by $Y_s$ the vector field whose component functions are homogeneous polynomials of degree $s$, and such that the $s-$jets of $Y$ and $Y_s$ are the same. As in the normal form theorem, define the map
  \eq{
  \left[ Y_s,- \right]_k:\mathcal H^k\to \mathcal H^{k+s-1}.
  }

  For $k>s$, the splitting of the space $\mathcal H^k$ is $\mathcal H^k=B^k+G^k$, where $B^k=\im\left([Y_s,H] \right)$, with now $H\in \mathcal H^{k-s+1}$. In this way, the conclusion of the normal form theorem remains valid by replacing the $Y'$ from above by
  \eq{
  Y'=Y_s+g_{s+1}+\ldots+g_\ell+R_\ell.
  }

  \end{remark}
\section*{Acknowledgements}
The authors are grateful with Robert Roussarie, David Chillingworth, and the anonymous reviewers for helpful discussions and comments that improved the text. HJK is partially supported by a CONACyT graduate grant.

\end{document}